\documentclass[11pt,reqno]{amsart}
\usepackage{hyperref}
\usepackage{graphics}
\usepackage{tikz}
\usepackage{amssymb}
\usetikzlibrary{patterns}
\usepackage{enumerate}
\newcommand{\di}{\displaystyle}
\newcommand{\tron}[1]{\left( #1 \right)}
\usepackage{color}
\usepackage[left=2.00cm, right=2.00cm, top=2.00cm, bottom=2.00cm]{geometry}
\numberwithin{equation}{section}
\newtheorem{theorem}{Theorem}[section]
\newtheorem{remark}{Remark}[section]

\newtheorem{lemma}[theorem]{Lemma}

\newtheorem{proposition}[theorem]{Proposition}

\allowdisplaybreaks
\newcommand{\e}{\varepsilon}
\newcommand{\R}{\mathbb{R}}
\newcommand{\Z}{\mathbb{Z}}

\newcommand{\ity}{\infty}
\newcommand{\dps}{\displaystyle}

\theoremstyle{plain}

\title{Blow-up and sharp lifespan estimates to the weakly coupled system of structurally damped wave equations with critical nonlinearities}
\begin{document}

\author{Trung Loc Tang}
\address{Trung Loc Tang \hfill\break
School of Mathematics and Computer Sciences, Hanoi National University of Education, 136 Xuan Thuy, Hanoi, Vietnam}
\author{Tuan Anh Dao$^*$}
\address{Tuan Anh Dao\hfill \break 
Faculty of Mathematics and Informatics, Hanoi University of Science and Technology, No.1 Dai Co Viet road, Hanoi, Vietnam}
\author{The Anh Cung}
\address{The Anh Cung\hfill \break
Department of Mathematics, Hanoi National University of Education, 136 Xuan Thuy, Hanoi, Vietnam}
\begin{abstract}
   In this paper, we would like to study the weakly coupled system of semilinear structurally damped wave equations with moduli of continuity in nonlinear terms whose powers belong to the critical curve in the $p-q$ plane. Our main purpose is to find a sharp condition for these moduli of continuity by investigating the global (in time) existence of small data Sobolev solutions and the blow-up result for solutions in finite time as well. Furthermore, when the blow-up phenomenon occurs, we are going to achieve the sharp lifespan estimates for the local (in time) Sobolev solution.
\end{abstract}
\keywords{Structural damping; Weakly coupled system; Modulus of continuity; Blow-up; Lifespan}

\maketitle

\section{Introduction}
\quad The semilinear damped wave equations with different nonlinearities have been widely studied by many authors and various results for the existence or non-existence of global (in time) solutions have been investigated. One of the well-known models is the following
\begin{equation}\label{sys1}
\begin{cases}
u_{tt}(t,x)-\Delta u(t,x) +u_t(t,x)=|u(t,x)|^p,\quad &x\in \mathbb{R}^n, \,t\geq 0,\\
u(0,x)=u_0(x), \quad u_t(0,x)=u_1(x), \quad &x\in \mathbb{R}^n,
\end{cases}
\end{equation}
with a power exponent $p>1$. Todorova-Yordanov in \cite{Todorova2001} proved that the solution exists globally when $p>1+2/n$ and blows up in finite time if $1<p\leq 1+2/n$ for $n\geq 3$. From this observation, we can say that the quantity $p_{\text{Fuj}}(n):=1+2/n$, the so-called Fujita exponent, is understood as the critical exponent of \eqref{sys1}. However, the power nonlinearity $\{|u(t,x)|^p\}_{p>1}$ seems too rough to verify the critical exponent of \eqref{sys1}, so the following Cauchy problem is of interest:
\begin{equation}\label{sys3}
\begin{cases}
u_{tt}(t,x)-\Delta u(t,x) +u_t(t,x)=|u(t,x)|^{1+\frac{2}{n}}\mu(|u(t,x)|),\quad &x\in \mathbb{R}^n,\, t\geq 0, \\
u(0,x)=u_0(x),\quad u_t(0,x)=u_1(x), \quad &x\in \mathbb{R}^n.
\end{cases}
\end{equation}
Here the function $\mu=\mu(|u(t,x)|)$ stands for the modulus of continuity, which is added to increase the regularity of the nonlinearity in a comparison between \eqref{sys3} and the origin equation \eqref{sys1}. The pioneering results are presented in \cite{EbeGirRei2020} by Ebert-Girardi-Reissig. They investigated the sharp condition for the nonlinearity to attain both the global (in time) existence of small data Sobolev solutions and the blow-up solution in finite time. More precisely, by denoting the integral of a modulus of continuity
\[
I_\mu:= \int_{0}^{c}\frac{\mu(s)}{s}\;ds
\]
with a sufficiently small constant $c>0$, the authors proved that if $I_\mu< \ity$, the Sobolev solution globally exists, meanwhile, the solution blows up in finite time when $I_\mu= \ity$ (see also \cite{TangDuong2026} for a further contribution involving a kind of this result). Afterwards, Dao-Reissig in \cite{AnhRei2021} considered the similar problem to \eqref{sys3} but for the weakly coupled system of semilinear damped wave equations, namely
\begin{equation}\label{sys5}
\begin{cases}
u_{tt}(t,x)-\Delta u(t,x) +u_t(t,x)=|v(t,x)|^{p^*}\mu_{1}(|v(t,x)|),\quad &x\in \mathbb{R}^n,\, t\geq 0, \\
v_{tt}(t,x)-\Delta v(t,x) +v_t(t,x)=|u(t,x)|^{q^*}\mu_{2}(|u(t,x)|),\quad &x\in \mathbb{R}^n,\, t\geq 0,\\
u(0,x)=u_0(x), \quad u_t(0,x)=u_1(x), \quad &x\in \mathbb{R}^n,\\
v(0,x)=v_0(x), \quad v_t(0,x)=v_1(x), \quad &x\in \mathbb{R}^n,
\end{cases}
\end{equation}
where $p^*,q^*$ belong to the following critical curve:
\begin{equation}\label{critcurve1}
\Gamma(p,q,n):= \frac{1+\max\{p,q\}}{pq-1}-\frac{n}{2}=0
\end{equation}
and $\mu_1=\mu_1(|v(t,x)|)$, $\mu_2=\mu_2(|u(t,x)|)$ present two suitable moduli of continuity. We call \eqref{critcurve1} the critcal curve because in terms of studying \eqref{sys5} without any moduli of continuity, i.e. $\mu_1\equiv \mu_2\equiv 1$, one realizes that a unique global (in time) solution exists when $\Gamma(p,q,n)<0$ and a blow-up result in finite time happens when $\Gamma(p,q,n)\ge 0$ (see \cite{Sun2007,NishiharaWakasugi}). By introducing the integral of two moduli of continuity
\begin{equation*}\label{con2}
I_{\mu_1,\mu_2}:= \di\int_{0}^c\frac{1}{s}\tron{\mu_1(s)}^{\frac{q^*}{q^*+1}}\tron{\mu_2(s)}^{\frac{1}{q^*+1}}\;ds,
\end{equation*}
the authors in \cite{AnhRei2021} found out that the system \eqref{sys5} witnesses the existence of a global (in time) Sobolev solution in the case $I_{\mu_1,\mu_2}<\infty$ and a blow-up phenomenon when $I_{\mu_1,\mu_2}=\infty$. Quite recently, Chen-Dao in \cite{ChenAnh2021} obtained the sharp lifespan estimates for \eqref{sys5} in the case $\mu_1\equiv \mu_2\equiv 1$, specifically, 
\[
T_{\varepsilon}\sim 
\begin{cases}
    \exp\tron{C\varepsilon^{-p^*-1}} &\text{ if }p^*=q^*,\\
    \exp\tron{C\varepsilon^{-\max\left\{\frac{p^*(p^*q^*-1)}{p^*+1},\frac{q^*(p^*q^*-1)}{q^*+1}\right\}}} &\text{ if }p^*\neq q^*,
\end{cases}
\]
in which $\varepsilon$ stands for a small, positive constant to describe the size of the initial data. Here, the lifespan $T_\varepsilon$ of solutions is understood as the quantity defined by
\begin{align*}\label{Lifespan_Defn}
T_\varepsilon:= \sup& \left\{T\in (0,\ity) : \mbox{There exists a unique local (in time) solution } (u,v) \mbox{ to \eqref{sys5} on [0,T)}\right.\notag \\
	&\left. \qquad\qquad\qquad \mbox{with a fixed parameter }\varepsilon>0\right\}.	
\end{align*}

In 2014, the other version of \eqref{sys1}, the so-called the semilinear structurally damped wave equations, has been investigated by D'Abbicco-Reissig in \cite{AbbRei2014}, namely
\begin{equation}\label{sys2}
\begin{cases}
u_{tt}(t,x)-\Delta u(t,x) +(-\Delta)^{\sigma}u_t(t,x)=|u(t,x)|^p,\quad &x\in \mathbb{R}^n,\, t\geq 0, \\
u(0,x)=u_0(x),\quad u_t(0,x)=u_1(x), \quad &x\in \mathbb{R}^n,
\end{cases}
\end{equation}
for any $\sigma\in [0,1]$. They proved that the solution to \eqref{sys2} exists globally when the exponent fulfills the following conditions:
$$p>
\left\{
\begin{aligned}
    &1+\dfrac{2}{n-2\sigma} &\text{ if }\sigma\in[0,1/2], \\
    &1+\dfrac{1+2\sigma}{n-1} &\text{ if }\sigma\in[1/2,1].
\end{aligned}
\right.
$$
Then, a blow-up result for $p<1+2/(n-2\sigma)$ was included in \cite{Anhblowup}, so we can claim that $p_{\text{crit}}:=1+2/(n-2\sigma)$ is sharp with $\sigma\in[0,1/2]$. Recently, Dao in \cite{Dao} (see more \cite{Anh2020}) worked with a weakly couple systems of semilinear structurally damped wave equations with the nonlinearities $|u(t,x)|^{q}, |v(t,x)|^{p}$ as follows:
\begin{equation}\label{sys4}
\begin{cases}
u_{tt}(t,x) -\Delta u(t,x) +(-\Delta)^{\sigma}u_t(t,x)=|v(t,x)|^p,\quad &x\in \mathbb{R}^n,\, t\geq 0, \\
v_{tt}(t,x) -\Delta v(t,x) +(-\Delta)^{\sigma}v_t(t,x)=|u(t,x)|^q,\quad &x\in \mathbb{R}^n,\, t\geq 0,\\
u(0,x)=u_0(x), \quad u_t(0,x)=u_1(x), \quad &x\in \mathbb{R}^n,\\
v(0,x)=v_0(x), \quad v_t(0,x)=v_1(x), \quad &x\in \mathbb{R}^n.
\end{cases}
\end{equation}
He showed that when the initial data belong to $L^1$, the critical curve of \eqref{sys4} is
\begin{equation}\label{critcurve}
\Gamma(p,q,n,\sigma):= \frac{1+\max\{p,q\}}{pq-1}=\frac{n-2\sigma}{2}.
\end{equation}
It means if $\Gamma(p,q,n,\sigma)<0$, there exists a unique global (in time) Sobolev solution, meanwhile, if $\Gamma(p,q,n,\sigma)\ge 0$, the solution blows up in finite time. The fact is that it still remains an open question to verify whether there exist global (in time) solutions or not on the critical curve \eqref{critcurve} when $\sigma$ is assumed to be any fractional number.

Motivated by those above papers, {\textit{the main interest of our work is to establish both the blow-up result and the sharp lifespan estimates for a weakly coupled system of semi-linear structurally damped wave equations with moduli of continuity in power nonlinear terms when the condition \eqref{critcurve} holds.}} For this reason, let us consider the following system:
\begin{equation}
\begin{cases} \label{eqsys}
    u_{tt}(t,x) -\Delta u(t,x) +(-\Delta)^{\sigma}u_t(t,x)=|v(t,x)|^{p^*}\mu_1(|v(t,x)|),\quad &x\in \mathbb{R}^n,\, t\geq 0,\\
    v_{tt}(t,x) -\Delta v(t,x) +(-\Delta)^{\sigma}v_t(t,x)=|u(t,x)|^{q^*}\mu_2(|u(t,x)|),\quad &x\in \mathbb{R}^n,\, t\geq 0,\\
    u(0,x)= u_0(x),\quad u_t(0,x)= u_1(x),\quad &x\in \mathbb{R}^n,\\
    v(0,x)= v_0(x),\quad v_t(0,x)= v_1(x),\quad &x\in \mathbb{R}^n,
\end{cases}
\end{equation}
where the functions $\mu_1=\mu_1(|v(t,x)|)$ and $\mu_2=\mu_2(|u(t,x)|)$ are some moduli of continuity. We assume that $\sigma\in [0,1/2]$ and $p^*, q^*>1$ belong to the critical curve described by \eqref{critcurve}. The main reason to only consider $\sigma\in [0,1/2]$ comes from the paper \cite{AbbRei2014} in which the authors have proposed to distinguish between ``parabolic like models'' in the case $\sigma\in [0,1/2]$, the so-called effective damping, and ``hyperbolic like models'' or ``wave like models'' in the case $\sigma\in (1/2,1]$, the so-called noneffective damping. To the best of authors' knowledge, it seems that nobody has ever succeeded to really determine the critical exponent for \eqref{sys2} in the context of noneffective damping. From this observation, it is quite natural to restrict our concern in terms of studying (\ref{eqsys}) with the situation of effective damping only in this paper. Without loss of generality, if we assume $p^*\leq q^*$, then the critical curve \eqref{critcurve} in the $p-q$ plane becomes
\begin{equation}\label{critcurve2}
    \frac{1+q^*}{p^*q^*-1}=\frac{n-2\sigma}{2}.
\end{equation}
It is worth pointing out that our approach used in this paper works well not only in the integer case $\sigma=0$, which has been explored in \cite{AnhRei2021}, but also in the fractional range $\sigma\in (0,1/2]$ (see more Remark \ref{Remark*}). To deal with the latter case, we need to construct new test functions without compact supports to treat the fractional Laplacian $(-\Delta)^{\sigma}$ appearing in (\ref{eqsys}). Throughout this work, one can recognize that our results give positive answers to several remaining quesions arising in the previous papers, especially, with respect to catching sharp estimates for lifespan of solutions (see Remark \ref{Interesting.Remark}, later). This is really the novelty of our paper. \medskip

\textbf{Notations} \medskip

We use the following notations throughout this paper.
\begin{itemize}
\item[$\bullet$] We write $f\lesssim g$ when there exists a constant $C>0$ such that $f\le Cg$, and $f \approx g$ when $g\lesssim f\lesssim g$.
\item[$\bullet$] As usual, the spaces $H^a$ and $\dot{H}^a$ with $a \ge 0$ stand for Bessel and Riesz potential spaces based on $L^2$ spaces. We denote $\widehat{f}(t,\xi):= \mathfrak{F}_{x\rightarrow \xi}\big(f(t,x)\big)$ as the Fourier transform with respect to the space variable of a function $f(t,x)$.
\item[$\bullet$] For a given number $s \in \R$, we denote $[s]:= \max \big\{k \in \Z : k\le s \big\}$ as its integer part.
\item[$\bullet$] We put $|x|:= \sqrt{x_1^2+x_2^2+\cdots+ x_n^2}$, the norm of $x \in \R^n$.
\item[$\bullet$] Finally, we introduce the initial data space $((u_0,u_1,(v_0,v_1))\in \mathcal{D}:=\tron{(H^{[n/2]+1}\cap L^1)\times(H^{[n/2]}\cap L^1)}^2$ with the corresponding norm
\begin{align*}          
\mathcal{D}[u_0,u_1,v_0,v_1] &:=\|u_0\|_{H^{[n/2]+1}}+\|u_0\|_{L^1}+\|u_1\|_{H^{[n/2]}}+\|u_1\|_{L^1}+\|v_0\|_{H^{[n/2]+1}}+\|v_0\|_{L^1}+\|v_1\|_{H^{[n/2]}}+\|v_1\|_{L^1}.
\end{align*}
\end{itemize}

Our main results for the case $p^*\leq q^*$, provided that $p^*,q^*$ belong to the curve \eqref{critcurve2}, read as follows. 

\begin{theorem}[\textbf{Blow-up}] \label{thr3.1}
    Let $\sigma\in \left[0,1/2\right]$ and $n>2\sigma$. We assume that the initial data $u_0,v_0,u_1,v_1\in L^1 $ satisfy
    \begin{equation}\label{con7}
        \int_{\mathbb{R}^n}\left(u_0(x)+u_1(x)\right)dx>0\quad \text{and}\quad \int_{\mathbb{R}^n}\left(v_0(x)+v_1(x)\right)dx>0.
    \end{equation}
    Moreover, we suppose the following assumptions of modulus of countinuity:
    \begin{equation}\label{con8}
        s^{k}\mu^{(k)}_{j}(s)= o\big(\mu_j(s)\big)\text{ as }s\to +0\text{ with }j,k=1,2
    \end{equation}
    and
    \begin{equation}\label{con9}
    \int_{0}^c\frac{1}{s}\tron{\mu_1(s)}^{\frac{q^*}{q^*+1}}\tron{\mu_2(s)}^{\frac{1}{q^*+1}}ds=\infty,
    \end{equation}
    where $c>0$ is a suitable small constant. Then, there is no global (in time) Sobolev solution to \eqref{eqsys} in the class
    $$ (u,v) \in \big(\mathcal{C}([0,\infty),L^2)\big)^2. $$
\end{theorem}
\begin{remark}
{\rm
    We give some examples of moduli of continuity $\mu_1$ and $\mu_2$ satisfying the condition (\ref{con9}) in Theorem \ref{thr3.1}. 
    \begin{enumerate}[1.]
    \item $\mu_1(0)=0$ and $\mu_1(s)=\left(\log \frac{1}{s}\right)^{-\alpha_1}$ with $\alpha_1>0$, $\mu_2(0)=0$ and $\mu_2(s)=\left(\log \frac{1}{s}\right)^{-\alpha_2}$ with $\alpha_2>0$, provided that
    $$
    \frac{q^*}{q^*+1} \alpha_1+\frac{1}{q^*+1} \alpha_2 \leq 1 ;
    $$
    \item $\mu_1(0)=0 \text { and } \mu_1(s)=\left(\log \frac{1}{s}\right)^{-1}\left(\log \log \frac{1}{s}\right)^{-1} \cdots\Bigg(\underbrace{\log \cdots \log \frac{1}{s}}_{m \text { times } \log }\Bigg)^{-\alpha_1}\text { with } m \in \mathbb{N}, \alpha_1>0,$ \\
    $\mu_2(0)=0 \text { and } \mu_2(s)=\left(\log \frac{1}{s}\right)^{-1}\left(\log \log \frac{1}{s}\right)^{-1} \cdots\Bigg(\underbrace{\log \cdots \log \frac{1}{s}}_{m \text { times } \log }\Bigg)^{-\alpha_2} \text { with } m \in \mathbb{N}, \alpha_2>0,$ \\
    \text { provided that } \\
    \[
    \frac{q^*}{q^*+1} \alpha_1+\frac{1}{q^*+1} \alpha_2 \leq 1 .
    \]
    \end{enumerate}    
}
\end{remark}

To state the second main result regarding sharp lifespan estimates for solutions to \eqref{eqsys}, let us replace the initial data $\left(u_0(x),u_1(x),v_0(x),v_1(x)\right)$ in \eqref{eqsys} by $\e\left(u_0(x),u_1(x),v_0(x),v_1(x)\right)$, provided that $\varepsilon$ stands for any small, positive constant, as well as introduce the following function:
    \begin{equation}\label{eq2sec4}
    \di \psi= \psi(R)=\int^{R}_{R_0}\frac{1}{r}\tron{\mu_1(C_0 r^{-\frac{n}{2}+\sigma})}^{\frac{q^*}{q^*+1}}\tron{\mu_2(C_0 r^{-\frac{n}{2}+\sigma})}^{\frac{1}{q^*+1}}\;dr,
    \end{equation}
    where $R_0,C_0$ are sufficiently large, positive constants independent of $\e$.
\begin{theorem}[\textbf{Sharp lifespan estimates}]
    Let $\sigma\in \left[0,1/2\right]$ and $n$ satisfies
    \[
    \begin{cases}
        4\sigma< n\leq 2 &\text{ if }\sigma\in \left[0,1/2\right),\\
        n=2 &\text{ if }\sigma= 1/2.
    \end{cases}
    \]
    We assume that the initial data $((u_0,u_1,(v_0,v_1))\in \mathcal{D}$ fulfill the relation \eqref{con7} and the following assumptions of modulus of continuity hold:
        \begin{equation}\label{con1thr4.1}
        s^{k}\mu^{(k)}_{j}(s)=\mathcal{O}\big(\mu^{(k-1)}_j(s)\big)\text{ as }s\to +0\text{ with }j,k=1,2
    \end{equation}
    and \eqref{con9}. Especially, when $p^*=q^*$, a further assumption required is that $s\in (0,c]\to \dfrac{\mu_1(s)}{\mu_2(s)}$ is a decreasing function in which $c>0$ is a suitable small constant independent of $\varepsilon$. Then, there exists a positive constant $\varepsilon_0$ such that for any $\varepsilon\in (0,\varepsilon_0]$ the lifespan $T_{\varepsilon}$ of Sobolev solutions
    $$ (u,v)\in \left(\mathcal{C}\left([0, \infty),H^1\cap L^{\infty}\right)\right)^2 $$
    to \eqref{eqsys} enjoys the following estimates:
    \[
    T_{\varepsilon}\sim
    \begin{cases}
    \tron{\psi^{-1}\tron{C\varepsilon^{-(p^*-1)}}}^{1-\sigma} &\text{ if $p^*=q^*$},\\
    \tron{\psi^{-1}\tron{C\varepsilon^{-\frac{q^*(p^*q^*-1)}{1+q^*}}}}^{1-\sigma} &\text{ if $p^*<q^*$},
    \end{cases}
    \]
    where $C$ is positive constant depending only on $\sigma,n,\mathcal{D}[u_0,u_1,v_0,v_1],\mu_1,\mu_2$ and $ \psi^{-1}$ stands for the inverse function of $\psi$ defined in \eqref{eq2sec4}.
\end{theorem}

\begin{remark} \label{Interesting.Remark}
{\rm
We want to point out that the continuous, increasing property of the function $\psi$ is to guarantee the existence of its inverse function $\psi^{-1}$. In addition, our result covers many classes of equations.
\begin{enumerate}[1.]
    \item If $\sigma=0$ in \eqref{eqsys}, we obtain the sharp lifespan estimate for the weakly coupled system of semilinear damped wave equations with critical nonlinearity as follows:
    $$ T_\e \sim 
    \left\{
    \begin{aligned}
        &\psi^{-1}\tron{C\e^{-(p^*-1)}}&\text{ if }p^*=q^*,\\
        &\psi^{-1}\tron{C\e^{-\max\left\{\frac{p^*(p^*q^*-1)}{p^*+1},\frac{q^*(p^*q^*-1)}{q^*+1}\right\}}}&\text{ if }p^*\neq q^*.
    \end{aligned}
    \right.
    $$
    \item If $u\equiv v$ in \eqref{eqsys}, we obtain the sharp lifespan estimate for semilinear structurally damped wave equations with critical nonlinearity as follows:
    $$ 
    T_\e \sim \left[\psi^{-1}\tron{C\e^{-(p^*-1)}}\right]^{1-\sigma}.
    $$
    \item If $\mu_1\equiv\mu_2\equiv 1$ and $\sigma=0$ in \eqref{eqsys}, we recover the sharp lifespan estimate for the weakly coupled system of semilinear damped wave equations obtained in \cite{ChenAnh2021} as follows:
    $$ T_\e \sim 
    \left\{
    \begin{aligned}
        &\exp\tron{C\e^{-(p^*-1)}}&\text{ if }p^*=q^*,\\
        &\exp\tron{C\e^{-\max\left\{\frac{p^*(p^*q^*-1)}{p^*+1},\frac{q^*(p^*q^*-1)}{q^*+1}\right\}}}&\text{ if }p^*\neq q^*.
    \end{aligned}
    \right.
    $$
    \item If $\mu_1\equiv \mu_2\equiv 1$ and $\sigma\in (0,1/2]$ in \eqref{eqsys}, we obtain the sharp lifespan estimate for the weakly coupled system of semilinear structurally damped wave equations as follows:
    $$ T_\e \sim 
    \left\{
    \begin{aligned}
        &\left[\exp\tron{C\e^{-(p^*-1)}}\right]^{1-\sigma}&\text{ if }p^*=q^*,\\
        &\left[\exp\tron{C\e^{-\max\left\{\frac{p^*(p^*q^*-1)}{p^*+1},\frac{q^*(p^*q^*-1)}{q^*+1}\right\}}}\right]^{1-\sigma} &\text{ if }p^*\neq q^*.
    \end{aligned}
    \right.
    $$
\end{enumerate}

}    
\end{remark}
\section{Blow-up in finite time}
\subsection{Preliminaries}
In this section, let us briefly sketch the result for the global (in time) existence of small data solutions to \eqref{eqsys}. To establish this, we have in mind the corresponding linear models of \eqref{eqsys} as follows:
\begin{equation}\label{eqlin}
\begin{cases}
    w_{tt}(t,x) -\Delta w(t,x) +(-\Delta)^{\sigma}w_t(t,x)=0, \quad &x\in \mathbb{R}^n,\, t\geq 0,\\
    w(0,x)=w_0(x), \quad w_t(0,x)=w_1(x) ,\quad &x\in \mathbb{R}^n.
\end{cases}
\end{equation}

\begin{proposition}\label{prop2.1} Let $\sigma \in\left[0, 1/2\right]$ in \eqref{eqlin}. The Sobolev solutions to \eqref{eqlin} satisfy the $\left(L^1 \cap L^2\right)-L^2$ estimates
\begin{align}
\left\|w(t, \cdot)\right\|_{L^2} &\lesssim(1+t)^{-\frac{n}{4(1-\sigma)}}\left\|w_0\right\|_{L^1 \cap L^2}+(1+t)^{-\frac{n}{4(1-\sigma)}+\frac{\sigma}{1-\sigma}}\left\|w_1\right\|_{L^1 \cap L^2}, \label{eq1prop2.1} \\
\left\|\nabla w(t, \cdot)\right\|_{L^2} &\lesssim(1+t)^{-\frac{n}{4(1-\sigma)}-\frac{1}{2(1-\sigma)}}\left\|w_0\right\|_{L^1 \cap H^1}+(1+t)^{-\frac{n}{4(1-\sigma)}-\frac{1-2\sigma}{2(1-\sigma)}}\left\|w_1\right\|_{L^1 \cap L^2}, \label{eq2prop2.1}
\end{align}
and the $(L^1\cap L^2)-L^{\infty}$ estimate
\begin{equation}\label{eq3prop2.1}
\left\|w(t, \cdot)\right\|_{L^\infty} \lesssim(1+t)^{-\frac{n}{2(1-\sigma)}}\left\|w_0\right\|_{L^1 \cap H^{[n/2]+1}}+(1+t)^{-\frac{n}{2(1-\sigma)}+\frac{\sigma}{1-\sigma}}\left\|w_1\right\|_{L^1 \cap H^{[n/2]}},
\end{equation}
for any $n$ being subject to
\[
\begin{cases}
    n>4\sigma &\text{ if }\sigma\in \left[0,1/2\right), \\
    n\geq 1 &\text{ if }\sigma=1/2.
\end{cases}
\]
\end{proposition}
\begin{proof}
    Since the first two estimates \eqref{eq1prop2.1}, \eqref{eq2prop2.1} have been proved in Lemma 1 in \cite{Mat1976} and Proposition 1 in \cite{Dao}, we only sketch the proof of the last one. Note that the application of the partial Fourier transform to \eqref{eqlin} gives the following expression:
    \[
    \widehat{w}(t,\xi)= \left(\mathcal{K}_0(t,\xi)+\frac{|\xi|^{2\sigma}}{2}\mathcal{K}_1(t,\xi)\right)\widehat{w_0}(\xi)+ \mathcal{K}_1(t,\xi)\widehat{w_1}(\xi),
    \]
    where
\[
\mathcal{K}_0(t,\xi)=
\begin{cases}
e^{\frac{-|\xi|^{2 \sigma}}{2} t} \cosh \left(\dfrac{|\xi|^{2 \sigma} \sqrt{1-4|\xi|^{2-4 \sigma}}}{2} t\right) &\text{if }\sigma<1/2\text{ and } |\xi| \leq 2^{-(1-2 \sigma)^{-1}}, \\
e^{\frac{-|\xi|^{2 \sigma}}{2} t} \cos \left(\dfrac{|\xi|^{2 \sigma} \sqrt{4|\xi|^{2-4 \sigma}-1}}{2} t\right) &\text{if }\sigma<1/2 \text{ and } |\xi|>2^{-(1-2 \sigma)^{-1}}, \\
e^{-\frac{|\xi|}{2} t} \cos \left(\dfrac{|\xi|\sqrt{3}}{2} t\right) &\text{if }\sigma = 1/2,
\end{cases}
\]
and
\[
\mathcal{K}_1(t,\xi)=
\begin{cases}
\dfrac{2 e^{\frac{-|\xi|^{2 \sigma}}{2} t}}{|\xi|^{2 \sigma} \sqrt{1-4|\xi|^{2-4\sigma}}} \sinh \left(\dfrac{|\xi|^{2 \sigma} \sqrt{1-4|\xi|^{2-4 \sigma}}}{2} t\right) &\text{if }\sigma<1/2 \text{ and } |\xi| \leq 2^{-(1-2 \sigma)^{-1}}, \\
\dfrac{2 e^{\frac{-|\xi|^{2 \sigma}}{2}} t}{|\xi|^{2 \sigma} \sqrt{4|\xi|^{2-4 \sigma}-1}} \sin \left(\dfrac{|\xi|^{2 \sigma} \sqrt{4|\xi|^{2-4 \sigma}-1}}{2} t\right) &\text{if }\sigma<1/2 \text{ and } |\xi|>2^{-(1-2\sigma)^{-1}},\\
\dfrac{2 e^{-\frac{|\xi|}{2} t}}{|\xi|\sqrt{3}} \sin \left(\dfrac{|\xi|\sqrt{3}}{2} t\right) &\text{if }\sigma = 1/2.
\end{cases}
\]
Let $\chi_k= \chi_k(r)$ with $k\in\{\rm L,H\}$ be smooth cut-off functions having the following properties:
\begin{align*}
&\chi_{\rm L}(r)=
\begin{cases}
1 &\quad \text{ if }r\leq \epsilon/2 \\
0 &\quad \text{ if }r\geq \epsilon
\end{cases}
\text{ and } \qquad
\chi_{\rm H}(r)= 1 -\chi_{\rm L}(r),
\end{align*}
where $\epsilon$ is a sufficiently small positive constant. We follow the same treatment as in the proof of Lemma 1 in \cite{Mat1976} to estimate solutions for high frequencies as follows:
\begin{align*}
&\left\|\mathfrak{F}^{-1}\tron{\tron{\mathcal{K}_0(t,\xi)+2^{-1}|\xi|^{2\sigma}\mathcal{K}_1(t,\xi)}\widehat{w_0}(\xi)\chi_{\rm H}(|\xi|)}\right\|_{L^{\infty}}\\
&\qquad \leq \left\|\tron{\mathcal{K}_0(t,\xi)+2^{-1}|\xi|^{2\sigma}\mathcal{K}_1(t,\xi)}\widehat{w_0}(\xi)\chi_{\rm H}(|\xi|)\right\|_{L^{1}}\\
&\qquad \lesssim e^{-ct}\||\xi|^{2\sigma-1}\widehat{w_0}(\xi)\chi_{\rm H}(|\xi|)\|_{L^1}\\
&\qquad \lesssim e^{-ct}\tron{\,\,\int_{|\xi|>\epsilon}|\xi|^{-2[n/2]-4+4\sigma}d\xi}^{1/2}\tron{\,\,\int_{|\xi|>\epsilon}|\xi|^{2[n/2]+2}\left|\widehat{w_0}(\xi)\right|^2d\xi}^{1/2} \lesssim e^{-ct}\left\|w_0\right\|_{H^{[n/2]+1}}
\end{align*}
and 
\begin{align*}
&\left\|\mathfrak{F}^{-1}\tron{\mathcal{K}_1(t,\cdot)\widehat{w_1}(\xi)\chi_{\rm H}(|\xi|)}\right\|_{L^{\infty}}\leq \left\|\mathcal{K}_1(t,\cdot)\widehat{w_1}(\xi)\chi_{\rm H}(|\xi|)\right\|_{L^{1}}\\
&\qquad \lesssim e^{-ct}\tron{\,\,\int_{|\xi|>\epsilon}|\xi|^{-2[n/2]-2}d\xi}^{1/2}\tron{\,\,\int_{|\xi|>\epsilon}|\xi|^{2[n/2]}\left|\widehat{w_1}(\xi)\right|^2d\xi}^{1/2} \lesssim e^{-ct}\left\|w_1\right\|_{H^{[n/2]}},
\end{align*}
where $c$ is a suitable positive constant. For low frequencies, we use the result of Proposition 4.1 in \cite{D'Abbicco2017} with $q=\infty$, $q_0=q_1=1$, $b=k=0$ to obtain
\begin{align*}
\left\|\mathfrak{F}^{-1}\tron{\widehat{w}(t,\xi)\chi_{\rm L}(|\xi|)}\right\|_{L^{\infty}}\lesssim (1+t)^{-\frac{n}{2(1-\sigma)}}\left\|w_0\right\|_{L^1}+(1+t)^{-\frac{n}{2(1-\sigma)}+\frac{\sigma}{1-\sigma}}\left\|w_1\right\|_{L^1}.
\end{align*}
Therefore, combining these estimates for high and low frequencies we may establish the desired estimate.
\end{proof}

By the aid of Proposition \ref{prop2.1} and repeating several proof steps in \cite{AnhRei2021} with minor modifications, we may conclude the following result for the global (in time) existence of small data solutions.
\begin{theorem}\label{thr2.1}
    Let $\sigma\in \left[0,1/2\right]$ and $n$ satisfies
    \[
\begin{cases}
    4\sigma<n\leq 2 &\text{ if }\sigma\in \left[0,1/2\right),\\
    n= 2 &\text{ if }\sigma=1/2.
\end{cases}
    \]
    Assume that the following assumption of modulus of continuity holds:
    \begin{equation}\label{con1thr2.1}
        s\mu'_{j}(s)=\mathcal{O}(\mu_j(s))\quad \text{as $s\to +0$ with $j=1,2$.}
    \end{equation}
    Moreover, we suppose that one of the following conditions is satisfied:
    \begin{itemize}
        \item[\rm i)] 
        \begin{equation}\label{con2thr2.1}
            \di\int_{0}^c\frac{\mu_1(s)}{s}\;ds<\infty \text{ and }\di\int_{0}^c\frac{\mu_2(s)}{s}\;ds<\infty.
        \end{equation}
        \item[\rm ii)] 
        \begin{equation}\label{con3thr2.1}
            \text{If }\di\int_{0}^c\frac{\mu_1(s)}{s}\;ds=\infty \text{ or } \di\int_{0}^c\frac{\mu_2(s)}{s}\;ds=\infty, \text{ then } \di\int_{0}^c\frac{1}{s}\tron{\mu_1(s)}^{\frac{q^*}{q^*+1}}\tron{\mu_2(s)}^{\frac{1}{q^*+1}}\;ds<\infty.
    \end{equation}
    \end{itemize}
    Here $c>0$ is a suitable small constant. Especially, when $q^*=p^*$, a further assumption required is that $s \in(0, c] \rightarrow \dfrac{\mu_1(s)}{\mu_2(s)}$ is a decreasing function. Then, there exists a constant $\varepsilon_0$ such that for any small data $\left(\left(u_0, u_1\right),\left( v_0, v_1\right)\right) \in \mathcal{D}$ satisfying the assumption $\mathcal{D}[u_0,u_1,v_0,v_1]\leq\varepsilon_0$, we have a uniquely determined global (in time) small data Sobolev solution
$$
(u, v) \in\left(\mathcal{C}\left([0, \infty),H^1\cap L^{\infty}\right)\right)^2.
$$
to \eqref{eqsys}. Additionally, the following estimates hold for $k=0,1$:
$$
\begin{aligned}
\left\|\nabla^{k} u(t, \cdot)\right\|_{L^2} & \lesssim (1+t)^{-\frac{n}{4(1-\sigma)}-\frac{k-2\sigma}{2(1-\sigma)}+s\left(p^*, q^*\right)}\gamma(t)\mathcal{D}[u_0,u_1,v_0,v_1],\\
\|u(t,\cdot)\|_{L^{\infty}}&\lesssim (1+t)^{-\frac{n}{2(1-\sigma)}+\frac{\sigma}{1-\sigma}+s(p^*,q^*)}\gamma(t)\mathcal{D}[u_0,u_1,v_0,v_1],\\
\left\|\nabla^{k} v(t, \cdot)\right\|_{L^2}&\lesssim (1+t)^{-\frac{n}{4(1-\sigma)}-\frac{k-2\sigma}{2(1-\sigma)}}\mathcal{D}[u_0,u_1,v_0,v_1],\\
\|v(t,\cdot)\|_{L^{\infty}}&\lesssim (1+t)^{-\frac{n}{2(1-\sigma)}+\frac{\sigma}{1-\sigma}}\mathcal{D}[u_0,u_1,v_0,v_1],
\end{aligned}
$$
where
\begin{equation}\label{con4thr2.1}
s\left(p^*, q^*\right):=\frac{q^*-p^*}{(1-\sigma)(p^* q^*-1)}
\end{equation}
and the weight function $\gamma=\gamma(t)$ is defined by
\begin{equation}\label{con5thr2.1}
\gamma(t):= \begin{cases}
1 & \text { if \eqref{con2thr2.1} holds}, \\
\left(\dfrac{\mu_1\left(c(1+t)^{-\ell}\right)}{\mu_2\left(c(1+t)^{-\ell}\right)}\right)^{\frac{1}{q^*+1}} & \text { if \eqref{con3thr2.1} holds},
\end{cases}
\end{equation}
with a sufficiently small constant $\ell>0$. 
\end{theorem}

\begin{remark}
{\rm
If we take $\sigma=0$, then we obtain the same results for the global (in time) existence of small data Sobolev solutions as Theorem 1.1 proved by Dao-Reissig in \cite{AnhRei2021}. Also, in this situation if $\mu_1\equiv\mu_2$ and $u=v$, then $p^*=q^*=1+2/(n-2\sigma)$ and our weakly coupled system can be considered as an equation. In this case, Theorem \ref{thr2.1} above recovers Theorem 3 in \cite{EbeGirRei2020}. 
}
\end{remark}
\begin{remark}
{\rm
Concerning a blow-up result, we can see that the case $\sigma=0$ was proved by Dao-Reissig in \cite{AnhRei2021}, where they used a kind of test function introduced in \cite{Ikeda2019}, however, for any $\sigma>0$ these test functions do not work so well. The difficulty comes from a nonlocal operator - the fractional Laplacian in the non-integer case of $\sigma$. For this reason, we are going to introduce a new kind of test functions which allows us to overcome this difficulty.}
\end{remark}

\subsection{Construction of test functions} \label{Testfucntions.Sec}
At first, we introduce the test function
\[
\varphi=\varphi(x):=
\begin{cases}
1 & \text{ if $|x|\leq 1$}, \\
\big(\sqrt[4]{1+(|x|-1)^4}\big)^{-1} & \text{ if $|x|\geq 1$},
\end{cases}
\]
and define
\[
\varphi^*= \varphi^*(x):=
\begin{cases}
0 &\text{ if $|x|<1$}, \\
\varphi(x) &\text{ if $|x|\geq 1$}.
\end{cases}
\]
\begin{lemma}[\cite{AnhFino}, Lemma 2.3]\label{lem9} Let $s\in (0,1]$. Then, the following estimate holds for all $x \in \mathbb{R}^n$:
$$
\left|(-\Delta)^{s}[\varphi(x)]^{-n-2s}\right| \lesssim [\varphi(x)]^{-n-2s}.
$$
\end{lemma}

\begin{lemma}\label{lem8} Let $q>0$. Then, the following estimate holds for any multi-index $\alpha$ satisfying $\alpha \geq 1$ and $x \in \mathbb{R}^n$:
$$
\left|\partial_x^\alpha [\varphi(x)]^{q}\right| \lesssim[\varphi(x)]^{q+\alpha}.
$$
\end{lemma}
\begin{proof}    
The case $|x|<1$ is trivial, so we consider only the case $|x|\geq 1$. For $|\alpha|\geq 1$, we have the following formula:
\[
\begin{aligned}
\partial_x^\alpha h\big(f(x)\big)=\sum_{k=1}^{|\alpha|} h^{(k)}(f(x))\left(\sum_{\substack{\gamma_1+\cdots+\gamma_k \leq \alpha \\
\left|\gamma_1\right|+\cdots+\left|\gamma_k\right|=|\alpha|,\left|\gamma_i\right| \geq 1}}\left(\partial_x^{\gamma_1} f(x)\right) \cdots\left(\partial_x^{\gamma_k} f(x)\right)\right),
\end{aligned}
\]
where $h=h(z)$ and $h^{(k)}(z)=\dfrac{d^k h(z)}{d z^k}$. If we choose $h(z)=z^{-\frac{q}{4}}$ and $f(x)=1+(|x|-1)^4$, we obtain
\begin{align*}
\left| \partial_x^\alpha [\varphi(x)]^{q}\right| &\leq \sum_{k=1}^{|\alpha|}\left(1+(|x|-1)^4\right)^{-\frac{q}{4}-k} \times\left(\sum_{\substack{\gamma_1+\cdots+\gamma_k \leq \alpha \\
\left|\gamma_1\right|+\cdots+\left|\gamma_k\right|=|\alpha|,\left|\gamma_i\right| \geq 1}}\left|\partial_x^{\gamma_1}\left(1+(|x|-1)^4\right)\right| \cdots\left|\partial_x^{\gamma_k}\left(1+(|x|-1)^4\right)\right|\right) \\
&\lesssim \sum_{k=1}^{|\alpha|}\left(1+(|x|-1)^4\right)^{-\frac{q}{4}-k} \times\left(\sum_{\substack{\gamma_1+\cdots+\gamma_k \leq \alpha \\
\left|\gamma_1\right|+\cdots+\left|\gamma_k\right|=|\alpha|,\left|\gamma_i\right| \geq 1}}(|x|-1)^{4-|\gamma_1|}\cdots(|x|-1)^{4-|\gamma_k|}\right) \\
&\lesssim \sum_{k=1}^{|\alpha|}\left(1+(|x|-1)^4\right)^{-\frac{q}{4}-k}(|x|-1)^{4k-|\alpha|}\\
&\lesssim [\varphi(x)]^{q+\alpha}\quad\text{ for any } |x|\geq 1.
\end{align*}
Hence, the proof is complete.
\end{proof}
\begin{lemma} \label{lem4}  
Let $s\in (0,1)$. Let $\varphi$ be a smooth function satisfying $\partial^2_x\varphi\in L^{\infty}$. For any $R>0$, let $\varphi_R$ be a function defined by
\[
\varphi_R(x):=\varphi(R^{-1/2}x)
\]
for all $x\in \mathbb{R}^n$. Then, $(-\Delta)^s(\varphi_R)$ satisfies the following scaling properties for all $x\in \mathbb{R}^n$:
\[
(-\Delta)^s(\varphi_R)(x)=R^{-s}\big((-\Delta)^{s}\varphi\big)(R^{-1/2}x).
\]
\end{lemma}
\begin{proof}
    Using a well-known argument (see \cite{Anhblowup}, Lemma 3) under the assumption $\partial^{2}_x\varphi\in L^{\infty}$, we may remove the principal value of the integral at the origin to conclude
$$
\begin{aligned}
(-\Delta)^s(\varphi_R)(x) &=-\frac{C_{n, s}}{2} \int_{\mathbb{R}^n} \frac{\varphi_R(x+y)+\varphi_R(x-y)-2 \varphi_R(x)}{|y|^{n+2 s}} d y \\
&=-\frac{C_{n, s}}{2 R^{s}} \int_{\mathbb{R}^n} \frac{\varphi\left(R^{-1/2} x+R^{-1/2} y\right)+\varphi\left(R^{-1/2} x-R^{-1/2} y\right)-2 \varphi\left(R^{-1/2} x\right)}{\left|R^{-1/2} y\right|^{n+2 s}} d\left(R^{-1/2} y\right) \\
&=R^{-s}\big((-\Delta)^s \varphi\big)\left(R^{-1/2} x\right).
\end{aligned}
$$
This completes our proof.
\end{proof}
Next, we introduce the functions $\eta(t) \in \mathcal{C}_0^{\infty}$ and $\eta^*(t)$, fulfilling 
\[
\eta(t)=
\begin{cases}
    1 &\text{ if }0\leq t\leq 1/2\\
    \text{decreasing} &\text{ if }1/2\leq t\leq 1\\
    0 &\text{ if }t\geq 1
\end{cases}
\quad \text{ and }\quad
\eta^*(t)=
\begin{cases}
0 &\text{ if }0\leq t<1/2\\
\eta(t) &\text{ if }t\geq 1/2.
\end{cases}
\]
It is clear to check that $\eta(t)$ is a decreasing function on $[0,\infty)$ and
\begin{equation}\label{as2}
\begin{aligned}
|\eta^{\prime}(t)|+|\eta^{\prime\prime}(t)|\lesssim C\text{ for any $t\in [0,\infty)$}.
\end{aligned}
\end{equation}
For $R$ is a large parameter in $[0,\infty)$, we introduce the functions
\[
\eta_R(t)=\eta\tron{R^{\sigma-1}t}\text{ and }\eta^*_{R}(t)=\eta^*\tron{R^{\sigma-1}t},
\]
and 
\[
\varphi_R(x)=\varphi\tron{R^{-1/2}x}\text{ and }\varphi^*_R(x)=\varphi^*\tron{R^{-1/2}x}.
\]
Now we define the test functions
\begin{equation}\label{testfunc}
\begin{aligned}
&\phi_{R}(x,t)=\left[\varphi_R(x)\right]^{n+2\sigma}[\eta_{R}(t)]^{\nu+2},\\
&\phi^1_{R}(x,t)=\left[\varphi^*_R(x)\right]^{n+2\sigma+2}[\eta_{R}(t)]^{\nu},\\
&\phi^2_{R}(x,t)=\left[\varphi_R(x)\right]^{n+2\sigma}[\eta^*_{R}(t)]^{\nu},
\end{aligned}
\end{equation}
where the parameter $\nu>0$ will be fixed later. Obviously, one sees that
\[
\begin{aligned}
&\text{supp}\phi_R(x,t)\subset Q_R:=\left\{(x,t):x\in \mathbb{R}^n \text{ and } t\in[0,R^{1-\sigma}]\right\} \\
&\text{supp}\phi^1_R(x,t)\subset Q^1_R:=\left\{(x,t):|x|\geq R^{1/2} \text{ and }t\in [0,R^{1-\sigma}]\right\},\\
&\text{supp}\phi^2_R(x,t)\subset Q^2_{R}:=\left\{(x,t):x\in \mathbb{R}^n \text{ and }t\in \left[\frac{1}{2}R^{1-\sigma}, R^{1-\sigma}\right]\right\}.
\end{aligned}
\]
Next, straightforward calculations gives
\begin{equation}\label{eq1sec3.1}
\begin{aligned}
|\partial_t \phi_R(x,t)|&=[\varphi_R(x)]^{n+2\sigma}\left|\partial_t[\eta_R(t)^{\nu+2}]\right|\lesssim R^{\sigma-1}[\varphi_R(x)]^{n+2\sigma}[\eta^*_R(t)]^{\nu+1}
\end{aligned}
\end{equation}
and
\begin{equation}\label{eq2sec3.1}
\begin{aligned}
|\partial^2_t\phi_R(x,t)|&=[\varphi_R(x)]^{n+2\sigma}\left|\partial^2_t[\eta_R(t)]^{\nu+2}\right|\\
&=[\varphi_R(x)]^{n+2\sigma}\left|(\nu+2)(\nu+1)[\eta^*_R(t)]^{\nu}[\eta^{\prime}_R(t)]^2+(\nu+2)[\eta^*_R(t)]^{\nu+1}[\eta^{\prime\prime}_R(t)]\right|\\
&\lesssim R^{2(\sigma-1)}[\varphi_R(x)]^{n+2\sigma}[\eta^*_R(t)]^{\nu}.
\end{aligned}
\end{equation}
Furthermore, using Lemma \ref{lem8} we arrive at
\begin{equation}\label{eq3sec3.1}
|\Delta \phi_R(x,t)|=\left|\eta_R(t)R^{-1}(\Delta [\varphi^*]^{n+2\sigma})(R^{-1/2}x)\right|\lesssim R^{-1}[\eta_R(t)]^{\nu+2}[\varphi_R^*(x)]^{n+2\sigma+2}.
\end{equation}
To deal with the fractional Laplacian operator, we employ Lemma \ref{lem9} and Lemma \ref{lem4} to estimate as follows:
\begin{align}\label{eq4sec3.1}
 |(-\Delta)^{\sigma}\partial_t\phi_R(x,t)| &\lesssim \left|R^{-1}\partial_t\eta^*_R(t)[\eta^*_R(t)]^{\nu+1}((-\Delta)^{\sigma}[\varphi]^{n+2\sigma})(R^{-1/2}x)\right| \notag \\
 &\lesssim R^{-1}[\eta^*_R(t)]^{\nu+1}[\varphi_R(x)]^{n+2\sigma}.   
\end{align}
\subsection{Proof of Theorem \ref{thr3.1}}
As we can see that the case $\sigma=0$ has been proved in \cite{AnhRei2021}, so we restrict ourselves to give a proof for the case $\sigma\in \left(0,1/2\right]$. For the ease of reading, let us firstly sketch our proof into several main steps as follows:
\begin{itemize}
    \item Step 1: Introduce two functionals $I_{R}$, $J_{R}$ and obtain early estimates.
    \item Step 2: Define two functions $\Phi_p$, $\Phi_q$ and catch some estimates combined with two functionals in Step 1.
    \item Step 3: Introduce auxiliary functions $g_q$, $g_p$, $G_p$, $G_q$ and establish their ordinary differential inequalities.
    \item Step 4: Estimate the achieved ordinary differential inequalities in Step 3 to complete the proof.
\end{itemize}
Next, we will present step by step in more details. \medskip

\noindent\textbf{$\bullet$ Step 1:} We define the following two functionals:
\[
\begin{aligned}
I_{R}&:=\int_{0}^{\infty}\int_{\mathbb{R}^n}|v(x,t)|^{p^*}\mu_1(|v(x,t)|)\phi_{R}(x,t)\;dxdt=\int_{Q_{R}}|v(x,t)|^{p^*}\mu_1(|v(x,t)|)\phi_{R}(x,t)d(x,t),\\    J_{R}&:=\int_{0}^{\infty}\int_{\mathbb{R}^n}|u(x,t)|^{q^*}\mu_2(|u(x,t)|)\phi_{R}(x,t)\;dxdt=\int_{Q_{R}}|u(x,t)|^{q^*}\mu_2(|u(x,t)|)\phi_{R}(x,t)d(x,t),   
\end{aligned}
\]
where the function $\phi_{R}(x,t)$ is defined as \eqref{testfunc}. Let us assume that
$$ (u,v)=\big(u(x,t),v(x,t)\big)\in \big(\mathcal{C}([0,\infty),L^2\big)^2 $$
is a global (in time) Sobolev solution to \eqref{eqsys}. We multiply the left-hand sides of \eqref{eqsys} by $\phi_R(x,t)$ and use the argument in \cite{Anhblowup} with the aid of Lemmas \ref{lem5}, \ref{lem6}, \ref{lem7} to achieve
\begin{equation}\label{esIR}
\begin{aligned}
0\leq I_{R}&=-\int_{\mathbb{R}^n}\left[u_1(x)\phi_R(x,0)-u_0\left(\partial_t\phi_R(0,x)-(-\Delta)^{\sigma}\phi_R(0,x)\right)\right]\;dx\\
&\quad +\int_{Q_{R}}u(x,t)\big(\partial^2_t \phi_{R}(x,t)-\Delta\phi_{R}(x,t)-(-\Delta)^{\sigma}\partial_t  \phi_{R}(x,t)\big)\;d(x,t)\\
&=:-\mathcal{D}^{\sigma}(u_0,u_1)+I^*_{R},
\end{aligned}
\end{equation}
\begin{equation}\label{esJR}
\begin{aligned}
0\leq J_{R}&=-\int_{\mathbb{R}^n}\left[v_1(x)\phi_R(x,0)-v_0\left(\partial_t\phi_R(0,x)-(-\Delta)^{\sigma}\phi_R(0,x)\right)\right]\;dx\\
&\quad +\int_{Q_{R}}v(x,t)\big(\partial^2_t \phi_{R}(x,t)-\Delta\phi_{R}(x,t)-(-\Delta)^{\sigma}\partial_t\phi_{R}(x,t)\big)\;d(x,t)\\
&=:-\mathcal{D}^{\sigma}(v_0,v_1)+J^*_{R}.
\end{aligned}
\end{equation}
To estimate $I_{R}^*$, $J_{R}^*$, we follow the method in \cite{EbeGirRei2020}, \cite{AnhRei2021}, \cite{AnhRei2019}. Recalling the estimates \eqref{eq2sec3.1}, \eqref{eq3sec3.1} and \eqref{eq4sec3.1}, we obtain
\begin{equation}\label{eq1sec3.2}
    |\partial^2_t  \phi_{R}(x,t)-\Delta\phi_{R}(x,t)-(-\Delta)^{\sigma}\partial_t  \phi_R(x,t)|\lesssim R^{-1}\big(\phi^1_{R}(x,t)+\phi^2_R(x,t)\big)
\end{equation}
for any $(x,t)\in Q_{R}$. Hence, we claim the following estimates:
    \begin{equation}\label{esIR*}
\begin{aligned}
\left|I_R^*\right| &\lesssim R^{-1}\tron{ \int_{Q^1_R}|u(x, t)|\phi^1_R(x,t)d(x,t)+\int_{Q^2_R}|u(x,t)|\phi^2_R(x,t)d(x, t)}
\end{aligned}
\end{equation}
and
\begin{equation}\label{esJR*}
\begin{aligned}
\left|J_R^*\right| &\lesssim  R^{-1}\tron{ \int_{Q^1_R}|u(x, t)|\phi^1_R(x,t)d(x,t)+\int_{Q^2_R}|u(x,t)|\phi^2_R(x,t)d(x, t)}.
\end{aligned}
\end{equation}

\noindent\textbf{$\bullet$ Step 2:} Let us introduce two functions $\Phi_p=\Phi_p(s)=s^{p^*} \mu_1(s)$ and $\Phi_q=\Phi_q(s)=s^{q^*} \mu_2(s)$. Then, for any $\delta <1$ we derive
\begin{equation}\label{eq2sec3.2}
\begin{aligned}
\Phi_q\left(|u(x, t)|\left(\phi^1_R(x, t)\right)^{1-\delta}\right) &=|u(x, t)|^{q^*}\left(\phi^1_R(x, t)\right)^{(1-\delta)q^*}\mu_2\left(|u(x, t)|\left(  \phi^1_R(x, t)\right)^{1-\delta}\right) \\
&\leq|u(x, t)|^{q^*}\left(  \phi^1_R(x, t)\right)^{(1-\delta)q^*} \mu_2(|u(x, t)|)=\Phi_q(|u(x, t)|)\left(  \phi^1_R(x, t)\right)^{(1-\delta)q^*},
\end{aligned}
\end{equation}
in which we have utilized the increasing property of $\mu_2=\mu_2(s)$ and the relation
$$
0 \leq  \phi^1_R(x, t)\leq 1.
$$
Similarly, one has
\begin{equation}\label{eq3sec3.2}
\begin{aligned}
 \Phi_p\left(|u(x, t)|\left(\phi^2_R(x, t)\right)^{1-\delta}\right) \leq \Phi_p(|u(x, t)|)\left( \phi^2_R(x, t)\right)^{(1-\delta)p^*}.
\end{aligned}
\end{equation}
From the assumption \eqref{con8}, we imply
$$
\Phi_q^{\prime \prime}(s)=s^{q^*-2}\left(q^*\left(q^*-1\right) \mu_2(s)+2 q^* s \mu_2^{\prime}(s)+s^2 \mu_2^{\prime \prime}(s)\right) \geq 0,
$$
that is, $\Phi_q$ is a convex function on a small interval $\left(0, c_0\right]$ with a sufficiently small constant $c_0>0$. Additionally, we can choose a convex continuation of $\Phi_q$ outside this interval to guarantee that $\Phi_q$ is convex on $[0, \infty)$. Applying the generalized Jensen's inequality from Lemma \ref{lem2} with $h(s)=\Phi_q(s), f(x, t)=|u(x, t)|\left(  \phi^1_R(x, t)\right), \eta \equiv (\phi^1_R(x,t))^{\delta}$ and $\Omega \equiv  Q^1_R$, we have 
$$
\Phi_q\left(\dfrac{\dps\int_{Q^1_{R}}|u(x, t)|\phi^1_R(x, t) d(x, t)}{\dps\int_{Q^1_{R}}\big(\phi^1_R(x,t)\big)^{\delta} d(x, t)}\right) \leq \dfrac{\dps\int_{Q^1_{R}} \Phi_q\left(|u(x, t)|\left(\phi^1_R(x, t)\right)^{1-\delta}\right)\big(\phi^1_R(x,t)\big)^{\delta}d(x, t)}{\dps\int_{Q^1_{R}}\big(\phi^1_R(x,t)\big)^{\delta}d(x, t)}.
$$
We approximate the integral of the functions $\tron{\phi^1_{R}(x,t)}^{\delta}$ over $Q_1$. Due to the fact
\begin{equation}\label{eq5sec3.1}
\begin{aligned}
    \int_{Q^1_{R}}[\phi^1_R(x,t)]^{\delta}d(x, t)=\int_{Q^1_{R}}[\eta_{R}(t))]^{\nu\delta}\left[\varphi^*_R(x)\right]^{(n+2\sigma+2)\delta}\;d(x,t),
\end{aligned}
\end{equation}
using the change of variable $\tilde{x}:=R^{-1/2}x$, $\tilde{t}:=R^{\sigma-1}t$, we derive
\begin{equation}\label{eq6sec3.1}
\begin{aligned}
\int_{Q^1_{R}}[\phi^1_R(x,t)]^{\delta}d(x, t)&=R^{\frac{n}{2}+1-\sigma}\int_{0}^{1}\int_{1}^{\infty}[\varphi^* (\tilde{x})]^{(n+2\sigma+2)\delta}\,|\tilde{x}|^{n-1}[\eta(\tilde{t})]^{\nu\delta}\;d|\tilde{x}| d\tilde{t} \approx R^{\frac{n}{2}+1-\sigma},
\end{aligned}
\end{equation} 
provided that the condition $n/(n+2\sigma+2)<\delta<1$ holds. It follows that
\begin{equation}\label{es2}
\Phi_q\left(\frac{\dps\int_{Q^1_{R}}|u(x, t)|\left(\phi^1_R(x, t)\right)  d(x, t)}{CR^{\frac{n}{2}+1-\sigma}}\right) \lesssim \frac{\dps\int_{Q^1_{R}} \Phi_q\left(|u(x, t)|\left(\phi^1_R(x, t)\right)^{1-\delta}\right)\tron{\phi^1_R(x,t)}^{\delta} d(x, t)}{R^{\frac{n}{2}+1-\sigma}}.
\end{equation}
From the estimates \eqref{eq2sec3.2}, one derives
\begin{equation}\label{eq4sec3.2}
\Phi_q\left(\frac{\dps\int_{Q^1_{R}}|u(x, t)|\left(\phi^1_R(x, t)\right)   d(x, t)}{CR^{\frac{n}{2}+1-\sigma}}\right) \lesssim \frac{\dps\int_{Q^1_{R}} \Phi_q(|u(x, t)|)\left(\phi^1_R(x, t)\right)^{q^*+(1-q^*)\delta}d(x, t)}{R^{\frac{n}{2}+1-\sigma}}.
\end{equation}
Because $\mu_2=\mu_2(s)$ is a strictly increasing function, $\Phi_q=\Phi_q(s)$ is also a strictly increasing function on $[0, \infty)$. As a result, it implies from \eqref{eq4sec3.2} that
\begin{equation}\label{eq5sec3.2}
\int_{Q^1_{R}}|u(x, t)|\left(\phi^1_R(x, t)\right) d(x, t)\lesssim R^{\frac{n}{2}+1-\sigma} \Phi_q^{-1}\left(\frac{\dps\int_{Q^1_{R}} \Phi_q(|u(x, t)|)\left(\phi^1_R(x, t)\right)^{q^*+(1-q^*)\delta} d(x,t)}{CR^{\frac{n}{2}+1-\sigma}}\right).
\end{equation}
For the domain $Q^2_R$, we also conclude for any $n/(n+2\sigma)<\delta<1$ that
\begin{equation}\label{eq7sec3.1}
    \int_{Q^2_{R}}[\phi^2_R(x,t)]^{\delta}d(x, t)\approx R^{\frac{n}{2}+1-\sigma}. 
\end{equation}
By the same argument as that used to prove \eqref{eq5sec3.2}, we obtain
\begin{equation}\label{eq6sec3.2}
\int_{Q^2_{R}}|u(x, t)|\left(\phi^2_R(x, t)\right) d(x, t)\lesssim R^{\frac{n}{2}+1-\sigma} \Phi_q^{-1}\left(\frac{\dps\int_{Q^2_{R}} \Phi_q(|u(x, t)|)\left(\phi^2_R(x, t)\right)^{q^*+(1-q^*)\delta} d(x,t)}{CR^{\frac{n}{2}+1-\sigma}}\right).
\end{equation}
Collecting the estimates \eqref{esIR}, \eqref{esIR*}, \eqref{eq5sec3.2} and \eqref{eq6sec3.2} with the fact that $\Phi^{-1}_q$ is concave function, we obtain
\begin{align}
&I_R+\mathcal{D}^{\sigma}(u_0,u_1)\lesssim R^{\frac{n}{2}-\sigma} \Phi_q^{-1}\left(\frac{\dps\int_{Q^1_{R}} \Phi_q(|u(x, t)|)\left(\phi^1_R(x, t)\right)^{q^*+(1-q^*)\delta}d(x,t)}{CR^{\frac{n}{2}+1-\sigma}}\right) \nonumber \\
&\hspace{3cm}+R^{\frac{n}{2}-\sigma} \Phi_q^{-1}\left(\frac{\int_{Q^2_{R}} \Phi_q(|u(x, t)|)\left(\phi^2_R(x, t)\right)^{q^*+(1-q^*)\delta}d(x,t)}{CR^{\frac{n}{2}+1-\sigma}}\right) \nonumber \\
&\qquad\lesssim R^{\frac{n}{2}-\sigma}\Phi_q^{-1}\left(\frac{\dps\int_{Q^1_{R}} \Phi_q(|u(x, t)|)\left(\phi^1_R(x, t)\right)^{q^*+(1-q^*)\delta}d(x,t)+\int_{Q^2_{R}} \Phi_q(|u(x, t)|)\left(\phi^2_R(x, t)\right)^{q^*+(1-q^*)\delta}d(x,t)}{CR^{\frac{n}{2}+1-\sigma}}\right). \label{eq7sec3.2}
\end{align}
Analogously, one gets
\begin{equation}\label{eq8sec3.2}
\begin{aligned}
&J_R+\mathcal{D}^{\sigma}(v_0,v_1)\\
&\qquad \lesssim R^{\frac{n}{2}-\sigma}\Phi_p^{-1}\left(\frac{\dps\int_{Q^1_{R}} \Phi_p(|v(x, t)|)\left(\phi^1_R(x, t)\right)^{p^*+(1-p^*)\delta} d(x, t)+\dps\int_{Q^2_{R}} \Phi_p(|v(x, t)|)\left(\phi^2_R(x, t)\right)^{p^*+(1-p^*)\delta} d(x, t)}{CR^{\frac{n}{2}+1-\sigma}}\right).
\end{aligned}
\end{equation}
We have the following lower bounds for $\mathcal{D}^{\sigma}(u_0,u_1)$ and $\mathcal{D}^{\sigma}(u_0,u_1)$ as follows:
\begin{align*}
    &\mathcal{D}^{\sigma}(u_0,u_1)\geq\int_{\mathbb{R}^n}\tron{u_0(x)+u_1(x)-CR^{-\sigma}|u_0(x)|}\phi_{R}(0,x) dx,\\
    &\mathcal{D}^{\sigma}(v_0,v_1)\geq\int_{\mathbb{R}^n}\tron{v_0(x)+v_1(x)-CR^{-\sigma}|v_0(x)|}\phi_{R}(0,x) dx.
\end{align*}
Since $u_0,u_1,v_0,v_1\in L^1$ and $\dps\lim_{R\to \infty}\phi_R(0,x)=1$, together with assumption \eqref{con7}, we have
\begin{align*}
    &\lim_{R\to \infty}\int_{\mathbb{R}^n}\tron{u_0(x)+u_1(x)}\phi_{R}(0,x)dx=\int_{\mathbb{R}^n}\tron{u_0(x)+u_1(x)}dx>0,\\
    &\lim_{R\to \infty}\int_{\mathbb{R}^n}\tron{u_0(x)+u_1(x)}\phi_{R}(0,x)dx=\int_{\mathbb{R}^n}\tron{u_0(x)+u_1(x)}dx>0,
\end{align*}
as well as
\[
\lim_{R\to \infty}R^{-\sigma}\int_{\mathbb{R}^n}|u_0(x)|\phi_R(0,x)|dx=0\text{ and }\lim_{R\to \infty}R^{-\sigma}\int_{\mathbb{R}^n}|v_0(x)|\phi_R(0,x)|dx=0.
\]
Therefore, we can choose $R_0$ is large enough such that for all $R\geq R_0$, the following relations are true:
\[
\mathcal{D}^{\sigma}(u_0,u_1)\geq \frac{1}{2}\int_{\mathbb{R}^{n}}\tron{u_0(x)+u_1(x)}d x>0\text{ and }\mathcal{D}^{\sigma}(v_0,v_1)\geq \frac{1}{2}\int_{\mathbb{R}^{n}}\tron{v_0(x)+v_1(x)}d x>0.
\] 
Combining this with \eqref{eq7sec3.2} and \eqref{eq8sec3.2}, for all $R\geq R_0$ we gain
\begin{align}
    I_R &\lesssim R^{\frac{n}{2}-\sigma} \Phi_q^{-1}\left(\frac{\dps\int_{Q^1_{R}} \Phi_q(|u(x, t)|)\left(\phi^1_R(x, t)\right)^{q^*+(1-q^*)\delta} d(x, t)+\dps\int_{Q^2_{R}} \Phi_q(|u(x, t)|)\left(\phi^2_R(x, t)\right)^{q^*+(1-q^*)\delta} d(x, t)}{CR^{\frac{n}{2}+1-\sigma}}\right) \label{eq9sec3.2}, \\
    J_R &\lesssim R^{\frac{n}{2}-\sigma} \Phi_p^{-1}\left(\frac{\dps\int_{Q^1_{R}} \Phi_p(|v(x, t)|)\left(\phi^1_R(x, t)\right)^{p^*+(1-p^*)\delta} d(x, t)+\dps\int_{Q^2_{R}} \Phi_p(|v(x, t)|)\left(\phi^2_R(x, t)\right)^{p^*+(1-p^*)\delta} d(x, t)}{CR^{\frac{n}{2}+1-\sigma}}\right).\label{eq10sec3.2}
\end{align}

\noindent\textbf{$\bullet$ Step 3:} For $\lambda>0$ and $s>0$, we define the following auxiliary functions:
$$
\begin{aligned}
& g_q=g_q(\lambda)=\int_{Q^1_{R}} \Phi_q(|u(x, t)|)\left(\phi^1_\lambda(x, t)\right)^{q^*+(1-q^*)\delta} d(x, t)+\int_{Q^2_{R}} \Phi_q(|u(x, t)|)\left(\phi^2_\lambda(x, t)\right)^{q^*+(1-q^*)\delta} d(x, t),\\
& g_p=g_p(\lambda)=\int_{Q^1_{R}} \Phi_p(|v(x, t)|)\left(\phi^1_\lambda(x, t)\right)^{p^*+(1-p^*)\delta} d(x, t)+\int_{Q^2_{R}} \Phi_p(|v(x, t)|)\left(\phi^2_\lambda(x, t)\right)^{p^*+(1-p^*)\delta} d(x, t),\\
&G_p=G_p(s)=\int_0^s g_p(\lambda)\lambda^{-1} d \lambda \quad \text{ and } \quad G_q=G_q(s)=\int_0^s g_q(\lambda)\lambda^{-1} d \lambda.
\end{aligned}
$$
Therefore, we can express
\[
\begin{aligned}
G_q(R)&=\int_0^R\left(\int_{Q^1_{R}} \Phi_q(|u(x, t)|)\left(\phi^1_\lambda(x, t)\right)^{q^*+(1-q^*)\delta} d(x, t)\right) \lambda^{-1} d \lambda+ \int_0^R\left(\int_{Q^2_{R}} \Phi_q(|u(x, t)|)\left(\phi^2_\lambda(x, t)\right)^{q^*+(1-q^*)\delta} d(x, t)\right) \lambda^{-1} d \lambda \\
&=\int_{Q^1_{R}} \Phi_q(|u(x, t)|)\tron{\int_{0}^{R}\tron{\left(\phi^1_\lambda(x, t)\right)^{q^*+(1-q^*)\delta}\lambda^{-1}\;d \lambda}}d(x, t)\\
&\quad+\int_{Q^2_{R}} \Phi_q(|u(x, t)|)\tron{\int_{0}^{R}\tron{\left(\phi^2_\lambda(x, t)\right)^{q^*+(1-q^*)\delta}\lambda^{-1}\;d \lambda} }d(x, t).
\end{aligned}
\]
Using the change of variables 
\[
\tilde{x}=\frac{x}{\lambda^{1/2}} \text{ and }\tilde{t}=\frac{t}{\lambda^{1-\sigma}},
\]
we have
\[
\begin{aligned}
    \text{supp}\big(\varphi^*(|\tilde{x}|)\big)^{\delta_1(n+2\sigma+2)}|\tilde{x}|^{-1}\subset [1,\infty)\quad \text{ and }\quad \text{supp}\big(\eta^*(\tilde{t})\big)^{\delta_1\nu}\tilde{t}^{-1}\subset [1/2,1]
\end{aligned}
\]
with a sufficiently small constant $\delta_1$, which will be chosen later. Since the function $\varphi$ is decreasing on $\left[0,\infty\right)$, the function $\eta$ is decreasing on $\left[0,1\right]$. Hence, we may conclude from the fact
$$
\frac{t}{\lambda^{1-\sigma}}\geq \frac{t}{R^{1-\sigma}}\quad \text{ and }\quad |x|\lambda^{-1/2}\geq |x|R^{-1/2}
$$
for $\lambda\in (0,R)$ that
$$
\phi^1_\lambda(x,t)\leq\phi^1_R(x,t) \text{ for } (x,t)\in Q^1_R\quad \text{ and } \quad \phi^2_\lambda(x,t)\leq\phi^2_R(x,t) \text{ for } (x,t)\in Q^2_R.
$$
Summarizing, we arrive at the estimate for $G_q(R)$ as follows:
$$
\begin{aligned}
G_q(R) & \lesssim \int_{Q^1_{R}} \Phi_q(|u(x, t)|)\left(\int_{0}^{R}\left(\phi^1_{\lambda}(x,t)\right)^{q^*+(1-q^*)\delta-\delta_1}\left(\phi^1_{\lambda}(x,t)\right)^{\delta_1}(\lambda)^{-1}d \lambda\right) d(x, t)\\
&+\int_{Q^2_{R}} \Phi_q(|u(x, t)|)\left(\int_{0}^{R}\left(\phi^2_{\lambda}(x,t)\right)^{q^*+(1-q^*)\delta-\delta_1}\left(\phi^2_{\lambda}(x,t)\right)^{\delta_1}(\lambda)^{-1}d \lambda\right) d(x, t)\\
&\lesssim \int_{Q^1_{R}} \Phi_q(|u(x, t)|)[  \phi^1_R(x,t)]^{q^*+(1-q^*)\delta-\delta_1}\tron{\int_{1}^{\infty}\left(\varphi^*(|\tilde{x}|)\right)^{\delta_1(n+2\sigma+2)}|\tilde{x}|^{-1}\;d|\tilde{x}|}\;d(x,t)\\
&\quad+\int_{Q^2_{R}} \Phi_q(|u(x, t)|)[  \phi^2_R(x,t)]^{q^*+(1-q^*)\delta-\delta_1}\tron{\int_{1/2}^{1}\left(\eta^*(\tilde{t})\right)^{\nu\delta_1}\tilde{t}^{-1}\;d\tilde{t}}\;d(x,t)\\
& \lesssim \int_{Q^1_{R}} \Phi_q(|u(x, t)|)[  \phi^1_R(x,t)]^{q^*+(1-q^*)\delta-\delta_1}d(x,t)+\int_{Q^2_{R}} \Phi_q(|u(x, t)|)[  \phi^2_R(x,t)]^{q^*+(1-q^*)\delta-\delta_1}d(x,t),
\end{aligned}
$$
where we used the relation
\[
0\leq \varphi_R(x),\, \eta_R(t)\leq 1.
\]
An analogous argument implies
\begin{align}
    G_p(R) &\lesssim \int_{Q^1_{R}} \Phi_p(|v(x, t)|)\left(\phi^1_{R}\left(t,x\right)\right)^{p^*+(1-p^*)\delta-\delta_1}d(x, t) +\int_{Q^2_{R}} \Phi_p(|v(x, t)|)\left(\phi^2_{R}\left(t,x\right)\right)^{p^*+(1-p^*)\delta-\delta_1}d(x, t).
\end{align}
Due to the fact $\phi^j_R(x,t)\leq \big(\phi_R(x,t)\big)^{\frac{\nu}{\nu+2}}$ with $j=1,2$, it follows that
\begin{equation}\label{eq11sec3.2}
G_q(R) \lesssim \int_{Q_{R}}|u(x, t)|^{q^*} \mu_2(|u(x, t)|)[\phi_{R}(x,t)]^{\frac{\nu}{\nu+2}\tron{q^*+(1-q^*)\delta-\delta_1}} d(x, t), 
\end{equation}
\begin{equation}\label{eq12sec3.2}
G_p(R) \lesssim \int_{Q_{R}}|v(x, t)|^{p^*} \mu_1(|v(x, t)|)[\phi_{R}(x,t)]^{\frac{\nu}{\nu+2}\tron{p^*+(1-p^*)\delta-\delta_1}}d(x, t).
\end{equation}
Since $q^*\geq p^*>1$, we can choose $\delta_1$ and $\nu$ satisfying
\[
    0<\delta_1<\delta+(1-\delta)p^*-1, \text{ that is, } p^*+(1-p^*)\delta-\delta_1-1>0,
\]
and then
$$ \nu\geq \frac{2}{p^*+(1-p^*)\delta-\delta_1-1}.$$
These conditions come to verify that
\begin{align}
G_q(R) &\lesssim \int_{Q_{R}}|u(x, t)|^{q^*} \mu_2(|u(x, t)|) \phi_{R}(x,t) d(x, t)\lesssim J_R, \label{eq13sec3.2} \\
G_p(R) &\lesssim \int_{Q_{R}}|v(x, t)|^{p^*} \mu_1(|v(x, t)|)\phi_{R}(x,t)d(x, t)\lesssim I_R. \label{eq14sec3.2}
\end{align}
Thus, the definitions of $G_q$ and $g_q$ lead to
$$
\begin{aligned}
\frac{d G_q}{d s}(s) s & =g_q(s), \text { i.e. }\left(\frac{d G_q}{d s}\right)(s=R) R=g_q(R), \\
\frac{d G_p}{d s}(s) s & =g_p(s), \text { i.e. }\left(\frac{d G_p}{d s}\right)(s=R) R=g_p(R),
\end{aligned}
$$
which imply
\begin{align}
    G_q^{\prime}(R) R=g_q(R) &=\int_{Q^1_{R}} \Phi_q(|u(x, t)|)\left(\phi^1_R(x, t)\right)^{q^*+(1-q^*)\delta} d(x, t) +\int_{Q^2_{R}} \Phi_q(|u(x, t)|)\left(\phi^2_R(x, t)\right)^{q^*+(1-q^*)\delta} d(x, t), \label{eq15sec3.2}
\end{align}
and 
\begin{align}
    G_p^{\prime}(R) R=g_p(R) &=\int_{Q^1_{R}} \Phi_p(|u(x, t)|)\left(\phi^1_R(x, t)\right)^{p^*+(1-p^*)\delta} d(x, t)  +\int_{Q^2_{R}} \Phi_p(|u(x, t)|)\left(\phi^2_R(x, t)\right)^{p^*+(1-p^*)\delta} d(x, t). \label{eq16sec3.2}
\end{align}
Plugging \eqref{eq15sec3.2} into \eqref{eq9sec3.2} and \eqref{eq16sec3.2} into \eqref{eq10sec3.2}, then combining them with the estimates \eqref{eq13sec3.2} and \eqref{eq14sec3.2}, respectively, one finds that
\[
\begin{aligned}
& G_q(R) \lesssim J_R \lesssim  R^{\frac{n}{2}-\sigma} \Phi_p^{-1}\left(\frac{G_p^{\prime}(R)}{CR^{\frac{n}{2}-\sigma}}\right), \\
& G_p(R)\lesssim I_R \lesssim  R^{\frac{n}{2}-\sigma}\Phi_q^{-1}\left(\frac{G_q^{\prime}(R)}{CR^{\frac{n}{2}-\sigma}}\right),
\end{aligned}
\]
which lead to
$$
\begin{aligned}
& \Phi_p\left(\frac{G_q(R)}{CR^{\frac{n}{2}-\sigma}}\right) \lesssim \frac{G_p^{\prime}(R)}{R^{\frac{n}{2}-\sigma}}, \\
& \Phi_q\left(\frac{G_p(R)}{CR^{\frac{n}{2}-\sigma}}\right) \lesssim \frac{G_q^{\prime}(R)}{R^{\frac{n}{2}-\sigma}}.
\end{aligned}
$$
Recalling the definition of the functions $\Phi_p$ and $\Phi_q$ we derive
$$
\begin{aligned}
& \left(\frac{G_q(R)}{CR^{\frac{n}{2}-\sigma}}\right)^{p^*} \mu_1\left(\frac{G_q(R) }{CR^{\frac{n}{2}-\sigma}}\right) \lesssim \frac{G_p^{\prime}(R)}{R^{\frac{n}{2}-\sigma}}, \\
& \left(\frac{G_p(R)}{CR^{\frac{n}{2}-\sigma}}\right)^{q^*} \mu_2\left(\frac{G_p(R)}{CR^{\frac{n}{2}-\sigma}}\right) \lesssim \frac{G_q^{\prime}(R)}{R^{\frac{n}{2}-\sigma}}.
\end{aligned}
$$
Thus, it follows that
$$
\begin{aligned}
& C^{-p^*}\theta R^{-({\frac{n}{2}-\sigma})(p^*-1)}\tron{G_q(R) }^{p^*}\mu_1\left(\frac{G_q(R) }{CR^{\frac{n}{2}-\sigma}}\right) \leq G_p^{\prime}(R), \\
& C^{-q^*}\theta R^{-({\frac{n}{2}-\sigma})(q^*-1)}\tron{G_p(R) }^{q^*}\mu_2\left(\frac{G_p(R)  }{CR^{\frac{n}{2}-\sigma}}\right) \leq G_q^{\prime}(R),
\end{aligned}
$$
for all $R \geq R_0$ and $\theta\in [0,1]$ is a parameter, which will be fixed later. Due to the increasing property of the functions $\mu_1=\mu_1(s), \mu_2=\mu_2(s), G_p=$ $G_p(R)$ and $G_q=G_q(R)$, the following inequalities hold:
$$
\begin{aligned}
& C^{-p^*}\theta R^{-({\frac{n}{2}-\sigma})(p^*-1)}\mu_1\left(\frac{G_q\left(R_0\right)}{CR^{\frac{n}{2}-\sigma}}\right)\left(G_q(R)\right)^{p^*} \lesssim G_p^{\prime}(R), \\
& C^{-q^*}\theta R^{-({\frac{n}{2}-\sigma})(q^*-1)} \mu_2\left(\frac{G_p\left(R_0\right)}{CR^{\frac{n}{2}-\sigma}}\right)\left(G_p(R)\right)^{q^*} \lesssim G_q^{\prime}(R).
\end{aligned}
$$
As a consequence, one has
$$
\begin{aligned}
& C^{-p^*}\theta R^{-({\frac{n}{2}-\sigma})(p^*-1)}\mu_1\left(C_0 R^{-\frac{n}{2}+\sigma}\right)\left(G_q(R)\right)^{p^*} \lesssim G_p^{\prime}(R), \\
& C^{-q^*}\theta R^{-({\frac{n}{2}-\sigma})(q^*-1)}\mu_2\left(C_0 R^{-\frac{n}{2}+\sigma}\right)\left(G_p(R)\right)^{q^*} \lesssim G_q^{\prime}(R),
\end{aligned}
$$
for all $R\geq R_0$, where $C_0=C_0\left(C, R_0\right):=C^{-1}\min \left\{G_p\left(R_0\right), G_q\left(R_0\right)\right\}$. Denoting

$$
\tau_1(\rho):=\frac{1}{{\rho}^{({\frac{n}{2}-\sigma})(p^*-1)}} \mu_1\left(C_0 \rho^{-\frac{n}{2}+\sigma}\right), \quad \tau_2(\rho):=\frac{1}{{\rho}^{({\frac{n}{2}-\sigma})(q^*-1)}}\mu_2\left(C_0 \rho^{-\frac{n}{2}+\sigma}\right),
$$
we obtain the following system of ordinary differential inequalities for $\tau \geq R_0$ :
\begin{equation}\label{eq17sec3.2}
G_p^{\prime}(R) \geq C_1\theta \tau_1(R)\left(G_q(R)\right)^{p^*}, \\
\end{equation}
\begin{equation}\label{eq18sec3.2}
G_q^{\prime}(R) \geq C_1 \theta\tau_2(R)\left(G_p(R)\right)^{q^*}.
\end{equation}

\noindent\textbf{$\bullet$ Step 4:} For any $R\geq R_0$, after multiplying \eqref{eq17sec3.2} by $G_q^{\prime}(\tau)$ and integrating by parts over $\left[R_0, R\right]$ we arrive at
\begin{equation}\label{eq19sec3.2}
\begin{aligned}
& G_p(R) G_q^{\prime}(R)-G_p\left(R_0\right) G_q^{\prime}\left(R_0\right)-\int^R_{R_0} G_p(\rho) G_q^{\prime \prime}(\rho) d \rho \\
&\quad \geq \frac{C_1\theta}{p^*+1} \tau_1(R)\left(G_q(R)\right)^{p^*+1}-\frac{C_1\theta}{p^*+1} \tau_1\left(R_0\right)\left(G_q\left(R_0\right)\right)^{p^*+1} -\frac{C_1\theta}{p^*+1} \int_{R_0}^R \tau_1^{\prime}(\rho)\left(G_q(\rho)\right)^{p^*+1} d \rho.
\end{aligned}
\end{equation}
To control the right-hand side (RHS) of \eqref{eq19sec3.2}, we can choose a sufficiently small constant $\theta=\theta\left(R_0\right)>0$ enjoying
\[
0<\theta\leq \min\left\{\frac{(p^*+1)G_p(R_0)G'_q(R_0)}{C_1\tau_1(R_0)\tron{G_q(R_0)}^{p^*+1}},1 \right\}
\]
to verify the inequality
$$
G_p\left(R_0\right) G_q^{\prime}\left(R_0\right)-\frac{C_1\theta}{p^*+1} \tau_1\left(R_0\right)\left(G_q\left(R_0\right)\right)^{p^*+1} \geq 0 .
$$
Hence, we achieve
$$
\text { RHS of }\eqref{eq19sec3.2} \geq \frac{C_1\theta}{p^*+1} \tau_1(R)\left(G_q(R)\right)^{p^*+1}-\frac{C_1\theta}{p^*+1} \int_{R_0}^R \tau_1^{\prime}(\rho)\left(G_q(\rho)\right)^{p^*+1} d \rho.
$$
By a direct calculation, we obtain the following equality:
$$
\tau_1^{\prime}(\rho)=-\left(\tron{\frac{n}{2}-\sigma}(p^*-1)+\frac{C_0}{\rho^{\frac{n}{2}-\sigma}}\frac{\mu_1^{\prime}\left(C_0 \rho^{-\frac{n}{2}+\sigma}\right)}{\mu_1\left(C_0 \rho^{-\frac{n}{2}+\sigma}\right)}\right)\frac{\tau_1(\rho)}{\rho}.
$$
Thanks to the assumption \eqref{con8} and the condition $n>2\sigma$, it is clear that $\tau^{\prime}_1(\rho)\leq 0$ for a sufficiently large constant $R_0$, which implies 
\begin{equation}\label{eq20sec3.2}
\text{RHS of }\eqref{eq19sec3.2}\geq \frac{C_1\theta}{p^*+1} \tau_1(R)\left(G_q(R)\right)^{p^*+1}.
\end{equation}
To estimate the left-hand side (LHS) of \eqref{eq19sec3.2}, we have
\[
\begin{aligned}
G_q^{\prime \prime}(\rho)=\tron{\frac{g_q(\rho)}{\rho}}^{\prime}&=\int_{Q^1_{R}} \Phi_q(|u(x, t)|)\left(\left(\phi^1_\rho(x, t)\right)^{q^*+(1-q^*)\delta}\rho^{-1}\right)^{\prime}d(x, t)\\
&\qquad +\int_{Q^2_{R}} \Phi_q(|u(x, t)|)\left(\left(\phi^2_\rho(x, t)\right)^{q^*+(1-q^*)\delta}\rho^{-1}\right)^{\prime}d(x, t).
\end{aligned}
\]
Now we will prove that $\dfrac{g_q(\rho)}{\rho}$ is an increasing function on the interval $[R_0,R/4]$ to obtain $G^{\prime\prime}_q(\rho)\geq 0$ for any $\rho \in [R_0,R/4]$. Regarding the integral over $Q^2_R$ one observes
\[
\rho^{\sigma-1}t\geq 4^{1-\sigma}R^{\sigma-1}t\geq \frac{4^{1-\sigma}}{2}\geq 1
\]
for $\rho\in \left[R_0,R/4\right]$ due to $(x,t)\in Q^2_R$. Then, we can conclude that
\[
\eta_{\rho}^*(t)=0\text{ and }\int_{Q^2_{R}} \Phi_q(|u(x, t)|)\left(\left(\phi^2_\rho(x, t)\right)^{q^*+(1-q^*)\delta}\rho^{-1}\right)^{\prime}d(x, t)=0\text{ for $\rho \in \left[R_0,R/4\right]$}.
\]
For the integral over $Q^1_R$, we only need to show that the function
\[
\begin{aligned}
\left[\varphi^*_{\rho}(x)\right]^{(n+2\sigma+2)(q^*+(1-q^*)\delta)}\rho^{-1}&=\tron{\frac{1}{\sqrt[4]{1+\tron{\frac{|x|}{\rho^{1/2}}-1}^4}\rho^{\frac{1}{(n+2\sigma+2)(q^*+(1-q^*)\delta)}}}}^{(n+2\sigma+2)(q^*+(1-q^*)\delta)}
\end{aligned}
\]
is an increasing function with $\rho\in \left[R_0,R/4\right]$. To achive this goal, we are going to verify that
\begin{equation}\label{eq21sec3.2}
\alpha(\rho):= \rho^{\frac{4}{(n+2\sigma+2)(q^*+(1-q^*)\delta)}}+\tron{\frac{|x|}{\rho^{1/2}}-1}^4\rho^{\frac{4}{(n+2\sigma+2)(q^*+(1-q^*)\delta)}}
\end{equation}
is decreasing. Indeed,
\begin{equation}\label{eq22sec3.2}
\begin{aligned}
\alpha'(\rho)= &\frac{4}{(n+2\sigma+2)(q^*+(1-q^*)\delta)}\rho^{\frac{4}{(n+2\sigma+2)(q^*+(1-q^*)\delta)}-1}-2\tron{\frac{|x|}{\rho^{1/2}}-1}^{3}\rho^{\frac{4}{(n+2\sigma+2)(q^*+(1-q^*)\delta)}}\frac{|x|}{\rho^{3/2}}\\
&\quad +\tron{\frac{|x|}{\rho^{1/2}}-1}^{4}\frac{4}{(n+2\sigma+2)(q^*+(1-q^*)\delta)}\rho^{\frac{4}{(n+2\sigma+2)(q^*+(1-q^*)\delta)}-1}<0,
\end{aligned}
\end{equation}
which is equivalent to
\begin{equation}\label{eq23sec3.2}
\begin{aligned}
&\frac{4}{(n+2\sigma+2)(q^*+(1-q^*)\delta)}\\
&\quad+\tron{\frac{|x|}{\rho^{1/2}}-1}^{3}\tron{\frac{2|x|}{\rho^{1/2}}\tron{\frac{2}{(n+2\sigma+2)(q^*+(1-q^*)\delta)}-1}-\frac{4}{(n+2\sigma+2)(q^*+(1-q^*)\delta)}}< 0.
\end{aligned}
\end{equation}
Since $(x,t)\in Q^1_R$ and $|x| \geq R^{1/2}$, we have $\rho\leq R/4$ and the relation
\[
\frac{|x|}{\rho^{1/2}}\geq 2\frac{|x|}{R^{1/2}}\geq 2
\]
combined with
\[
\frac{2}{(n+2\sigma+2)(q^*+(1-q^*)\delta)}-1<0\text{ for $n\geq 1$},
\]
to deduce \eqref{eq23sec3.2} immediately. So we conclude that for $\rho\in \left[R_0,R/4\right]$, the function $\dfrac{g_q(\rho)}{\rho}$ is increasing which entails
\begin{equation}\label{eq24sec3.2}
\int^{R/4}_{R_0}G_p(\rho)G_q^{\prime \prime}(\rho)\;d\rho\geq 0.
\end{equation}
From the fact
\[
G^{\prime \prime}_q(\rho)=\frac{g_q^{\prime}(\rho)-G^{\prime}_q(\rho)}{\rho}
\]
and $g_q^{\prime}(\rho)\geq 0$ since $g_q(\rho)$ is increasing, we gain the estimate
\begin{equation}\label{eq25sec3.2}
-\di\int^R_{R/4} G_p(\rho) G_q^{\prime \prime}(\rho) d\rho\leq\int^R_{R/4} \frac{G_p(\rho)G_q^{\prime}(\rho)}{\rho} d\rho.
\end{equation}
The function $G_q^{\prime}(\rho)\rho=g_q(\rho)$ is increasing, so
\[
\frac{G_q^{\prime}(\rho)}{\rho}\leq \frac{G_q^{\prime}(R)R}{\rho^2}\leq \frac{16G_q^{\prime}(R)}{R}\text{ for $\rho\in \left[R/4,R\right]$}.
\]
Thus, it follows that
\begin{equation}\label{eq26sec3.2}
\int^R_{R/4} \frac{G_p(\rho)G_q^{\prime}(\rho)}{\rho} d\rho\leq G_p(R)\int^{R}_{ R/4}\frac{G^{\prime}_q(\rho)}{\rho}\;d \rho\leq \frac{12G_p(R)G^{\prime}_q(R)}{R}R\lesssim G_p(R)G^{\prime}_q(R).
\end{equation}
From \eqref{eq25sec3.2} and \eqref{eq26sec3.2}, we obtain 
\begin{equation}\label{eq27sec3.2}
\text { LHS of }\eqref{eq19sec3.2} \lesssim G_p(R) G_q^{\prime}(R).
\end{equation}
Consequently, from \eqref{eq20sec3.2} and \eqref{eq27sec3.2}, it implies
$$
G_p(R) G_q^{\prime}(R) \geq C_2 \tau_1(R)\left(G_q(R)\right)^{p^*+1},
$$
that is,
\begin{equation}\label{eq28sec3.2}
G_p(R) \geq \frac{C_2 \tau_1(R)\left(G_q(R)\right)^{p^*+1}}{G_q^{\prime}(R)} .
\end{equation}
By plugging \eqref{eq28sec3.2} into \eqref{eq18sec3.2}, one gets
$$
G_q^{\prime}(R) \geq \frac{C_3 \tau_2(R)\left(\tau_1(R)\right)^{q^*}\left(G_q(R)\right)^{q^*\left(p^*+1\right)}}{\left(G_q^{\prime}(R)\right)^{q^*}},
$$
which is equivalent to
$$
\begin{aligned}
G_q^{\prime}(R) \geq & C_4\left(\tau_2(R)\right)^{\frac{1}{q^*+1}}\left(\tau_1(R)\right)^{\frac{q^*}{q^*+1}}\left(G_q(R)\right)^{\frac{q^*\left(p^*+1\right)}{q^*+1}} \\
& =\frac{C_4}{R}\left(\mu_1\left(C_0 R^{-\frac{n}{2}+\sigma}\right)\right)^{\frac{q^*}{q^*+1}}\left(\mu_2\left(C_0 R^{-\frac{n}{2}+\sigma}\right)\right)^{\frac{1}{q^*+1}}\left(G_q(R)\right)^{\frac{q^*\left(p^*+1\right)}{q^*+1}},
\end{aligned}
$$
where we have utilized the equality
\[
\tron{\tron{\frac{n}{2}-\sigma}(p^*-1)}\frac{q^*}{q^*+1}+\tron{\tron{\frac{n}{2}-\sigma}(q^*-1)}\frac{1}{q^*+1}=1.
\]
The above estimate implies
\begin{equation}\label{eq29sec3.2}
\frac{C_4}{R}\left(\mu_1\left(C_0 R^{-\frac{n}{2}+\sigma}\right)\right)^{\frac{q^*}{q^*+1}}\left(\mu_2\left(C_0 R^{-\frac{n}{2}+\sigma}\right)\right)^{\frac{1}{q^*+1}} \leq \frac{G_q^{\prime}(R)}{\left(G_q(R)\right)^{\frac{q^*\left(p^*+1\right)}{q^*+1}}} .
\end{equation}
Integrating two sides of \eqref{eq29sec3.2} over $[R_0,R]$ leads to
\[
\begin{aligned}
& C_4 \int_{R_0}^{R} \frac{1}{r}\left(\mu_1\left(C_0 r^{-\frac{n}{2}+\sigma}\right)\right)^{\frac{q^*}{q^*+1}}\left(\mu_2\left(C_0 r^{-\frac{n}{2}+\sigma}\right)\right)^{\frac{1}{q^*+1}} d r \\
& \quad \leq \int_{R_0}^{R} \frac{G_q^{\prime}(r)}{\left(G_q(r)\right)^{\frac{q^*\left(p^*+1\right)}{q^*+1}}} d r =-\left.\frac{q^*+1}{p^* q^*-1}\left(G_q(r)\right)^{-\frac{p^*q^*-1}{q^*+1}}\right|_{r=R_0} ^{r=R} \leq \frac{n-2\sigma}{2}\left(G_q\left(R_0\right)\right)^{-\frac{2}{n-2\sigma}},
\end{aligned}
\]
Let $R\to \infty$, we obtain
\[
C_4 \int_{R_0}^{\infty} \frac{1}{r}\left(\mu_1\left(C_0 r^{-\frac{n}{2}+\sigma}\right)\right)^{\frac{q^*}{q^*+1}}\left(\mu_2\left(C_0 r^{-\frac{n}{2}+\sigma}\right)\right)^{\frac{1}{q^*+1}} d r\leq \frac{n-2\sigma}{2}\left(G_q\left(R_0\right)\right)^{-\frac{2}{n-2\sigma}}.
\]
After using the change of variables $s=C_0 r^{-\frac{n}{2}+\sigma}$ and noticing the condition $n>2\sigma$, we derive
$$
C_4 \int_0^{C_0 R_0^{-\frac{n}{2}+\sigma}}\frac{1}{s}\left(\mu_1(s)\right)^{\frac{q^*}{q^*+1}}\left(\mu_2(s)\right)^{\frac{1}{q^*+1}} d s \leq \frac{n-2\sigma}{2}\left(G_q\left(R_0\right)\right)^{-\frac{2}{n-2\sigma}}.
$$
This contradicts to the assumption \eqref{con9}. Therefore, the proof of Theorem \ref{thr3.1} is completed.
\begin{remark} \label{Remark*}
{\rm
Let us explain the novelty of the proof of Theorem \ref{thr3.1}. When $\sigma=0$, the authors in \cite{AnhRei2021} used the test functions with compact support to establish their result. However, due to the limitations of this technique, it is not straightforward to adapt the previous method to the case, where $\sigma$ is not an integer, since $(-\Delta)^{\sigma}$ is a nonlocal operator. In our approach, the novelty lies in Step 2 and Step 4, where we have succeeded to handle the two non-compact domains $Q^1_R$ and $Q^2_R$. These domains arise from the nonlocal nature of the fractional operator $(-\Delta)^{\sigma}$ in the application of test functions without compact support.
}    
\end{remark}

\section{Sharp lifespan estimates}
As we can see in the previous sections, if the assumption \eqref{con9} holds, then following the same approach as in the proof of Theorem \ref{thr2.1} we claim that there exists a local (in time) solution
$$ (u,v)\in \big(\mathcal{C}([0,T],H^1\cap L^{\infty})\big)^2 $$
to \eqref{eqsys} and this solution blows up in finite time. An interesting question is whether we can successfully obtain sharp lifespan estimates for this solution. To give a positive answer for this question, let us recall function $\psi= \psi(R)$ introduced in \eqref{eq2sec4}. Notice that the assumption \eqref{con9} implies the condition $\psi(\tau)\to \infty$ if $\tau\to \infty$. 

\subsection{Estimates for upper bound}
By the same argument as in the proof of blow-up result, there exists a sufficiently large constant $R_0=R_0(r_0,r_1,n,\sigma,\mu_1,\mu_2)$ independent of $\varepsilon$ such that
\[
\mathcal{D}^{\sigma}(u_0,u_1)\geq \frac{\varepsilon}{2}\int_{\mathbb{R^n}}\big(u_0(x)+u_1(x)\big)dx\text{ and }\mathcal{D}^{\sigma}(v_0,v_1)\geq \frac{\varepsilon}{2}\int_{\mathbb{R^n}}\big(v_0(x)+v_1(x)\big)dx
\]
for all $R\geq R_0$. Without loss of generality, we can assume that $R^{1-\sigma}_0\leq R^{1-\sigma}\leq T_{\varepsilon} $. 
Recalling the estimates \eqref{eq7sec3.2} and \eqref{eq8sec3.2}, we achieve
\begin{equation}
\begin{aligned}
&I_R+\varepsilon\int_{\mathbb{R}^n}\big(u_0(x)+u_1(x)\big)\;dx\\
&\lesssim R^{\frac{n}{2}-\sigma}\Phi_p^{-1}\left(\frac{\dps\int_{Q^1_{R}} \Phi_p(|v(x, t)|)\left(\phi^1_R(x, t)\right)^{p^*+(1-p^*)\delta} d(x, t)+\dps\int_{Q^2_{R}} \Phi_p(|v(x, t)|)\left(\phi^2_R(x, t)\right)^{p^*+(1-p^*)\delta} d(x, t)}{CR^{\frac{n}{2}+1-\sigma}}\right).
\end{aligned}
\end{equation}
and
\begin{equation}
\begin{aligned}
&J_R+\varepsilon \int_{\mathbb{R}^n}\big(v_0(x)+v_1(x)\big)\;dx\\
&\lesssim R^{\frac{n}{2}-\sigma}\Phi_p^{-1}\left(\frac{\dps\int_{Q^1_{R}} \Phi_p(|v(x, t)|)\left(\phi^1_R(x, t)\right)^{p^*+(1-p^*)\delta} d(x, t)+\dps\int_{Q^2_{R}} \Phi_p(|v(x, t)|)\left(\phi^2_R(x, t)\right)^{p^*+(1-p^*)\delta} d(x, t)}{CR^{\frac{n}{2}+1-\sigma}}\right).
\end{aligned}
\end{equation}
Denoting 
\[
\mathcal{D}[u_0,u_1]=\int_{\mathbb{R}^n}\big(u_0(x)+u_1(x)\big)\;dx \quad\text{ and }\quad \mathcal{D}[v_0,v_1]=\int_{\mathbb{R}^n}\big(v_0(x)+v_1(x)\big)\;dx
\]
and using the same notation of $g_q(x)$, $G_q(x)$ as in Step 3 of the proof of Theorem \ref{thr3.1}, one finds
\[
\begin{aligned}
& G_q(R)+\varepsilon \mathcal{D}[u_0,u_1]\lesssim J_R+\varepsilon \mathcal{D}[u_0,u_1] \lesssim  R^{\frac{n}{2}-\sigma} \Phi_p^{-1}\left(\frac{G_p^{\prime}(R)}{CR^{\frac{n}{2}-\sigma}}\right), \\
& G_p(R)+\varepsilon \mathcal{D}[v_0,v_1]\lesssim I_R +\varepsilon \mathcal{D}[v_0,v_1]\lesssim  R^{\frac{n}{2}-\sigma}\Phi_q^{-1}\left(\frac{G_q^{\prime}(R)}{CR^{\frac{n}{2}-\sigma}}\right).
\end{aligned}
\]
Such these estimates lead to
\begin{equation}\label{eq2sec4.1}
    CR^{\frac{n(1-p^*)+ 2\sigma p^*}{1-\sigma}+1}\tron{G_q(R)+\varepsilon \mathcal{D}[u_0,u_1]}^{p^*}\mu_1\left(\frac{G_q(R)+\varepsilon \mathcal{D}[u_0,u_1]}{C_0R^{-\frac{n}{2}+\sigma}}\right) \leq G_p^{\prime}(R)
\end{equation}
and 
\begin{equation}\label{eq3sec4.1}
    CR^{\frac{n(1-q^*)+ 2\sigma q^*}{1-\sigma}+1}\tron{G_p(R)+\varepsilon \mathcal{D}[v_0,v_1]}^{q^*}\mu_2\left(\frac{G_p(R)+\varepsilon \mathcal{D}[v_0,v_1]}{C_0R^{-\frac{n}{2}+\sigma}}\right) \leq G_q^{\prime}(R).
\end{equation}
Since $\mathcal{D}[u_0,u_1]>0$ and $\mathcal{D}[v_0,v_1]>0$, repeating Step 4 in the proof of Theorem \ref{thr3.1} one arrives at
\begin{equation*}
\frac{C_2}{R}\left(\mu_1\left(C_1R^{-\frac{n}{2}+\sigma}\right)\right)^{\frac{q^*}{q^*+1}}\left(\mu_2\left(C_1 R^{-\frac{n}{2}+\sigma}\right)\right)^{\frac{1}{q^*+1}} \leq \frac{G_q^{\prime}(R)}{\left(G_q(R)\right)^{\frac{q^*\left(p^*+1\right)}{q^*+1}}} .
\end{equation*}
Next, we introduce the function
$$ \omega(s):=\psi^{-1}\tron{\frac{\psi(s)}{2}}, $$
where $\psi$ is defined as in \eqref{eq2sec4}. Then, it holds $\omega(s)\to \infty$ if $s\to \infty$ and $R_0\leq \omega(s)\leq s$ for any $s\in [R_0,\infty)$. Furthermore, we have the property
\begin{align}
    \omega(R)=\psi^{-1}\tron{\frac{\psi(R)}{2}} &\Leftrightarrow
    \psi(R)=2\psi\big(\omega(R)\big) \nonumber \\
    &\Leftrightarrow \frac{\psi(R)}{2}=\int_{R_0}^{\omega(R)}\frac{1}{r}\left(\mu_1\left(C_1 r^{-\frac{n}{2}+\sigma}\right)\right)^{\frac{p^*}{p^*+1}}\left(\mu_2\left(C_1 r^{-\frac{n}{2}+\sigma}\right)\right)^{\frac{1}{p^*+1}}\;dr \nonumber \\
    &\qquad\qquad =\int_{\omega(R)}^{R}\frac{1}{r}\left(\mu_1\left(C_1 r^{-\frac{n}{2}+\sigma}\right)\right)^{\frac{p^*}{p^*+1}}\left(\mu_2\left(C_1 r^{-\frac{n}{2}+\sigma}\right)\right)^{\frac{1}{p^*+1}}\;dr. \label{eq4sec4.1}
\end{align}
Integrating two side of \eqref{eq3sec4.1} over $\left[\omega(R),R\right]$ gives
\begin{align}
& C_2 \int_{\omega(R)}^{R} \frac{1}{r}\left(\mu_1\left(C_1 r^{-\frac{n}{2}+\sigma}\right)\right)^{\frac{q^*}{q^*+1}}\left(\mu_2\left(C_1 r^{-\frac{n}{2}+\sigma}\right)\right)^{\frac{1}{q^*+1}} d r \nonumber \\
& \quad \leq \int_{\omega(R)}^{R} \frac{G_q^{\prime}(r)}{\left(G_q(r)\right)^{\frac{q^*\left(p^*+1\right)}{q^*+1}}} d r =-\left.\frac{q^*+1}{p^* q^*-1}\left(G_q(r)\right)^{-\frac{p^*q^*-1}{q^*+1}}\right|_{r=\omega(R)} ^{r=R} \leq \frac{n-2\sigma}{2}\left(G_q\left(\omega(R)\right)\right)^{-\frac{2}{n-2\sigma}}. \label{eq5sec4.1}
\end{align}
\begin{itemize}
    \item \textbf{Case 1:} If $p^*<q^*$, since
    $$\frac{n(1-q^*)+ 2\sigma q^*}{1-\sigma}+1<-1$$
and $G_p(r),\mu_2(r)$ are increasing functions, integrating two sides of \eqref{eq3sec4.1} over $[R_0,\omega(R)]$ we may claim that
\begin{align}
G_q\big(\omega(R)\big)&\geq \varepsilon^{q^*}C(\mathcal{D}[v_0,v_1])^{q^*}\int_{R_0}^{\omega(R)} r^{\frac{n(1-q^*)+ 2\sigma q^*}{1-\sigma}+1}\mu_2\left(\frac{G_p(r)+\varepsilon \mathcal{D}[v_0,v_1]}{C_0 r^{-\frac{n}{2}+\sigma}}\right)dr \nonumber\\
&\geq \varepsilon^{q^*}C(\mathcal{D}[v_0,v_1])^{q^*}\mu_2\left(\frac{G_p(R_0)+\varepsilon \mathcal{D}[v_0,v_1]}{C_0 {R_0}^{-\frac{n}{2}+\sigma}}\right)\int_{R_0}^{\omega(R)} r^{\frac{n(1-q^*)+ 2\sigma q^*}{1-\sigma}+1}dr \nonumber\\
&\geq \varepsilon^{q^*}C_3\big(\mathcal{D}[v_0,v_1]\big)^{q^*}. \label{eq6sec4.1}
\end{align}
After plugging \eqref{eq6sec4.1} into \eqref{eq5sec4.1}, one realizes
\begin{equation}\label{eq7sec4.1}
\int_{\omega(R)}^{R} \frac{1}{r}\left(\mu_1\left(C_1 r^{-\frac{n}{2}+\sigma}\right)\right)^{\frac{q^*}{q^*+1}}\left(\mu_2\left(C_1 r^{-\frac{n}{2}+\sigma}\right)\right)^{\frac{1}{q^*+1}} d r\leq C_4 \varepsilon^{\frac{q^*(p^*q^*-1)}{q^*+1}}.
\end{equation}
Recalling the definition of $\psi$ we show that
\[
\psi(R)=2\int_{\omega(R)}^{R} \frac{1}{r}\left(\mu_1\left(C_1 r^{-\frac{n}{2}+\sigma}\right)\right)^{\frac{q^*}{q^*+1}}\left(\mu_2\left(C_1 r^{-\frac{n}{2}+\sigma}\right)\right)^{\frac{1}{q^*+1}}\;dr\leq 2C_4\varepsilon^{\frac{q^*(p^*q^*-1)}{q^*+1}}
\]
from \eqref{eq4sec4.1} and \eqref{eq7sec4.1}. Taking the function $\psi^{-1}$ in both sides of the previous estimate and then passing $R\to T_{\varepsilon}^{\frac{1}{1-\sigma}}$, we conclude
\[
T_{\varepsilon}\lesssim \tron{\psi^{-1}\tron{C\varepsilon^{\frac{q^*(p^*q^*-1)}{q^*+1}}}}^{1-\sigma}.
\]
    \item \textbf{Case 2:} If $p^*=q^*=1+2/(n-2\sigma)$, from \eqref{eq2sec4.1} and \eqref{eq3sec4.1} one sees
\[
\varepsilon^{p^*}\left(\mathcal{D}[u_0,u_1]\right)^{\frac{(p^*)^2}{p^*+1}}\left(\mathcal{D}[v_0,v_1]\right)^{\frac{p^*}{p^*+1}}R^{-1}\tron{\mu_1\tron{C_1R^{-\frac{n}{2}+\sigma}}}^{\frac{p^*}{p^*+1}}\tron{\mu_2\tron{C_1R^{-\frac{n}{2}+\sigma}}}^{\frac{1}{p^*+1}}\;ds\leq G^{\prime}_{p}(R).
\]
Hence, integrating two sides of the above inequality over $[R_0,\omega(R)]$ we obtain
\begin{equation}\label{eq8sec4.1}
\begin{aligned}
G_{p}(\omega(R))-G_{p}(R_0)&\geq C_5\varepsilon^{p^*}\int_{R_0}^{\omega(R)}\frac{1}{r}\left(\mu_1\left(C_1 r^{-\frac{n}{2}+\sigma}\right)\right)^{\frac{p^*}{p^*+1}}\left(\mu_2\left(C_1 r^{-\frac{n}{2}+\sigma}\right)\right)^{\frac{1}{p^*+1}}\;dr\\
&\geq C_5\varepsilon^{p^*}\psi\big(\omega(R)\big).
\end{aligned}
\end{equation}
From \eqref{eq5sec4.1} and \eqref{eq8sec4.1}, one entails
$$ \int_{\omega(R)}^{R}\frac{1}{r}\left(\mu_1\left(C_1 r^{-\frac{n}{2}+\sigma}\right)\right)^{\frac{p^*}{p^*+1}}\left(\mu_2\left(C_1r^{-\frac{n}{2}+\sigma}\right)\right)^{\frac{1}{p^*+1}}\;dr\leq C_6\e^{-\frac{2p^*}{n-2\sigma}}\big(\psi(\omega(R))\big)^{-\frac{2}{n-2\sigma}}. $$
The combination of this with \eqref{eq4sec4.1} gives
\begin{equation*}
\psi(R)\leq C_7\varepsilon^{-\frac{2p^*}{n-2\sigma}}\big(\psi(R)\big)^{-\frac{2}{n-2\sigma}} \Leftrightarrow \psi(R) \leq C_8\varepsilon^{-p^*+1}.
\end{equation*}
Finally, letting $R\to T_{\varepsilon}^{\frac{1}{1-\sigma}}$ we achieve
\[
T_{\varepsilon}\lesssim \tron{\psi^{-1}\tron{C\varepsilon^{-(p^*-1)}}}^{1-\sigma}.
\]
\end{itemize}

\subsection{Estimates for lower bound}
To prove the lower bound, we utilize some estimates in Proposition \ref{prop2.1}. The proof consists of two main steps. Firstly, we point out that the lifespan of solutions tends to infinity if the parameter $\varepsilon$ tends to zero. Then, we can suppose that the lifespan is greater than a sufficiently large constant for small $\varepsilon$. This assumption allows us to optimize the function $\gamma(t)$, defined in \eqref{con5thr2.1}, and finally leads to the lower bound of lifespan.  
\begin{proof} 
We define the following function:
\begin{equation} \label{Psi_function}
\Psi(t):=\int_{0}^{t}(1+\tau)^{-1}\tron{\mu_1\tron{c(1+\tau)^{-\ell}}}^{\frac{q^*}{q^*+1}}\tron{\mu_2\tron{c(1+\tau)^{-\ell}}}^{\frac{1}{q^*+1}}\;d\tau,
\end{equation}
where the parameter $\ell$ comes from the definition of $\gamma(t)$ in \eqref{con5thr2.1}. To prove the lower bounded of lifespan, we introduce the evolution spaces as follows:
\[
X_j(T)= \mathcal{C}\big([0,T],H^1\cap L^{\infty}\big) \text{ with $j=1,2$},
\]
endowed with their corresponding norms
\begin{align*}
    \|u\|_{X_1(T)} &:=\sup _{t \in[0, T]}\left((1+t)^{-s(p^*, q^*)}\gamma(t)^{-1}\big(\Psi(t)\big)^{\beta(p^*,q^*)} \mathcal{N}[u](t)\right), \\
\|v\|_{X_2(T)} &:=\sup _{t \in[0, T]}\Big(\mathcal{N}[v](t)\Big),
\end{align*}
where we define
$$ \beta(p^*,q^*):=\frac{(n-2\sigma)(p^*-1)}{2q^*}>0 $$
and 
$$
\mathcal{N}[w](t):=(1+t)^{\frac{n}{4(1-\sigma)}}\|w(t, \cdot)\|_{L^2}+(1+t)^{\frac{n}{4(1-\sigma)}-\frac{1-2\sigma}{2(1-\sigma)}}\|\nabla w(t, \cdot)\|_{L^2}+(1+t)^{\frac{n}{2(1-\sigma)}-\frac{\sigma}{1-\sigma}}\|w\|_{L^{\infty}}
$$
with $w=u$ or $w=v$. Then, we can introduce the solution space $X(T)$ of \eqref{eqsys} by
$$
X(T)= X_1(T) \times X_2(T)
$$
carrying the norm
$$
\|(u, v)\|_{X(T)}:=\|u\|_{X_1(T)}+\|v\|_{X_2(T)} .
$$
\textbf{$\bullet$ Case 1:} We consider the case $p^*<q^*$, hence from \eqref{con4thr2.1} we have $s(p^*,q^*)>0$. From the fact $\Psi(t)\leq C\ln(1+t)$ and by Remark 2.1 in \cite{AnhRei2021}, we deduce that
\begin{equation}\label{eq0sec4.2}
(1+t)^{-s(p^*, q^*)}\big(\gamma(t)\big)^{-1}\big(\Psi(t)\big)^{\beta(p^*,q^*)}\lesssim 1.
\end{equation}
Using Proposition \ref{prop2.1} and the above estimate, we have proved that
\[
\|(u^{\mathrm{ln}},v^{\mathrm{ln}})\|_{X(t)}\leq \varepsilon c_0(n,\mathcal{D}[u_0,u_1,v_0,v_1]).
\]
Following several steps in the proof of Theorem 1.1 in \cite{AnhRei2021} with small modifications we derive
\begin{align}
\|\nabla^{k}u^{\mathrm{nl}}(t,\cdot)\|_{L^2}&\lesssim(1+t)^{-\frac{n}{4(1-\sigma)}-\frac{k-2\sigma}{2(1-\sigma)}}\gamma(t)\|v\|_{X_2(T)}^{p^*} \nonumber \\
&\qquad \times \int_0^t(1+\tau)^{-1+s(p^*,q^*)}\left(\mu_1\left(C_1(1+\tau)^{-\ell}\right)\right)^{\frac{q^*}{q^*+1}}\left(\mu_2\left(C_1(1+\tau)^{-\ell}\right)\right)^{\frac{1}{q^*+1}}d\tau \nonumber \\
&\lesssim (1+t)^{-\frac{n}{4(1-\sigma)}-\frac{k-2\sigma}{2(1-\sigma)}}\|v\|^{p^*}_{X_2(T)}, \label{eq1sec4.2}, \\
\|u^{\mathrm{nl}}(t,\cdot)\|_{L^{\infty}}&\lesssim(1+t)^{-\frac{n}{2(1-\sigma)}+\frac{\sigma}{1-\sigma}}\gamma(t)\|v\|_{X_2(T)}^{p^*}\nonumber \\
&\qquad \times \int_0^t(1+\tau)^{-1+s(p^*,q^*)}\left(\mu_1\left(C_1(1+\tau)^{-\ell}\right)\right)^{\frac{q^*}{q^*+1}}\left(\mu_2\left(C_1(1+\tau)^{-\ell}\right)\right)^{\frac{1}{q^*+1}}d\tau\nonumber \\
&\lesssim (1+t)^{-\frac{n}{2(1-\sigma)}+\frac{\sigma}{1-\sigma}}\|v\|^{p^*}_{X_2(T)}, \label{eq2sec4.2}
\end{align}
and
\begin{align}
    \|\nabla^{k}v^{\mathrm{nl}}(t,\cdot)\|_{L^2}&\lesssim (1+t)^{-\frac{n}{4(1-\sigma)}-\frac{k-2\sigma}{2(1-\sigma)}}\|u\|_{X_1(T)}^{q^*}\nonumber \\
    &\qquad \times \int_0^t(1+\tau)^{-1}\left(\mu_1\left(C_1(1+\tau)^{-\ell}\right)\right)^{\frac{q^*}{q^*+1}}\left(\mu_2\left(C_1(1+\tau)^{-\ell}\right)\right)^{\frac{1}{q^*+1}}\big(\Psi(t)\big)^{-\beta(p^*,q^*)q^*}d\tau\nonumber \\
    &\lesssim (1+t)^{-\frac{n}{4(1-\sigma)}-\frac{k-2\sigma}{2(1-\sigma)}}\|u\|_{X_1(T)}^{q^*} \int_{0}^{t}\big(\Psi(t)\big)^{-\beta(p^*,q^*)q^*}\;d\big(\Psi(\tau)\big) \nonumber \\
    &\lesssim (1+t)^{-\frac{n}{4(1-\sigma)}-\frac{k-2\sigma}{2(1-\sigma)}}\big(\Psi(t)\big)^{1-\beta(p^*,q^*)q^*}\|u\|_{X_1(T)}^{q^*}. \label{eq3sec4.2} \\
    \|v^{\mathrm{nl}}(t,\cdot)\|_{L^\infty}&\lesssim (1+t)^{-\frac{n}{2(1-\sigma)}+\frac{\sigma}{1-\sigma}}\|u\|_{X_1(T)}^{q^*}\nonumber \\
    &\qquad \times \int_0^t(1+\tau)^{-1}\left(\mu_1\left(C_1(1+\tau)^{-\ell}\right)\right)^{\frac{q^*}{q^*+1}}\left(\mu_2\left(C_1(1+\tau)^{-\ell}\right)\right)^{\frac{1}{q^*+1}}\big(\Psi(t)\big)^{-\beta(p^*,q^*)q^*}d\tau\nonumber \\
    &\lesssim (1+t)^{-\frac{n}{2(1-\sigma)}+\frac{\sigma}{1-\sigma}}\big(\Psi(t)\big)^{1-\beta(p^*,q^*)q^*}\|u\|_{X_1(T)}^{q^*}. \label{eq4sec4.2}
\end{align}
Combining all estimates \eqref{eq1sec4.2}-\eqref{eq4sec4.2} leads to
\[
\begin{aligned}
(1+t)^{\frac{n}{4(1-\sigma)}+\frac{k-2\sigma}{2(1-\sigma)}-s(p^*,q^*)}\gamma(t)^{-1}\big(\Psi(t)\big)^{\beta(p^*,q^*)}\left\|\nabla^ku^{\mathrm{nl}}(t, \cdot)\right\|_{L^2}&\leq c^u_1\big(\Psi(t)\big)^{\beta(p^*,q^*)}\|v\|_{X_2(T)}^{p^*},\\
(1+t)^{\frac{n}{2(1-\sigma)}-\frac{\sigma}{1-\sigma}-s(p^*,q^*)}\gamma(t)^{-1}\big(\Psi(t)\big)^{\beta(p^*,q^*)}\|u^{\mathrm{nl}}(t,\cdot)\|_{L^{\infty}}&\leq c^{u}_2\big(\Psi(t)\big)^{\beta(p^*,q^*)}\|v\|^{p^*}_{X_2(T)},\\
(1+t)^{\frac{n}{4(1-\sigma)}+\frac{k-2\sigma}{2(1-\sigma)}}\left\|\nabla^kv^{\mathrm{nl}}(t, \cdot)\right\|_{L^2}&\leq c^v_1\big(\Psi(t)\big)^{1-\beta(p^*,q^*)q^*}\|u\|_{X_1(T)}^{q^*},\\
(1+t)^{\frac{n}{2(1-\sigma)}-\frac{\sigma}{1-\sigma}}\|v^{nl}(t,\cdot)\|_{L^{\infty}}&\leq c^{v}_2\big(\Psi(t)\big)^{1-\beta(p^*,q^*)q^*}\|u\|^{q^*}_{X_1(T)},
\end{aligned}
\]
where $c^u_1,c^u_2,c^v_1,c^v_2$ are some positive constants independent of $T$. From the above estimates and the definition of norm in $Y_1(T),Y_2(T)$, one arrives at
\begin{equation}\label{eq5sec4.2}
    \|u\|_{X_1(T)}\leq \varepsilon c_0+c^u\big(\Psi(T)\big)^{\beta(p^*,q^*)}\|v\|_{X_2(T)}^{p^*},
\end{equation}    
\begin{equation}\label{eq6sec4.2}
    \|v\|_{X_2(T)}\leq \varepsilon c_0+c^v\big(\Psi(T)\big)^{1-\beta(p^*,q^*)q^*}\|u\|_{X_1(T)}^{q^*},
\end{equation}
where $c^u$ and $c^v$ are some positive constants depending on $c_1^u,c_2^u$ and $c_1^v,c_2^v$ respectively. Afterwards, motivated by the approach in \cite{IkeOga2016}, we determine
$$ T^*:=\sup \left\{T \in\left[0, T_{\varepsilon}\right) \text { such that } F(T):=\|(u, v)\|_{X(T)} \leq M \varepsilon\right\} $$
with a sufficiently large constant $M>0$, which will be defined in next steps. Thanks to the fact $\|u\|_{X_1\left(T^*\right)} \leq\|(u, v)\|_{X\left(T^*\right)} \leq M \varepsilon$, it holds from \eqref{eq2sec4.2} that
\begin{equation}\label{eq8sec4.2}
\|v\|_{X_2\left(T^*\right)} \leq \varepsilon c_0+c^v M^{q^*}\big(\Psi(T^*)\big)^{1-\beta(p^*,q^*)q^*} \varepsilon^{q^*},
\end{equation}
provided that $1-\beta(p^*,q^*)q^*>0$. Then, substituting \eqref{eq8sec4.2} into \eqref{eq5sec4.2} one sees
$$
\begin{aligned}
\|u\|_{X_1\left(T^*\right)} & \leq \varepsilon c_0+\big(\Psi(T^*)\big)^{\beta(p^*,q^*)}\left(c_2 \varepsilon^{p^*}+c_3 M^{p^*q^*}\big(\Psi(T^*)\big)^{p^*(1-\beta(p^*,q^*)q^*)} \varepsilon^{p^*q^*}\right) \\
& \leq \varepsilon\left(c_0+c_2\big(\Psi(T^*)\big)^{\beta(p^*,q^*)} \varepsilon^{p^*-1}+c_3 M^{p^*q^*}\big(\Psi(T^*)\big)^{p^*-\beta(p^*,q^*)(p^*q^*-1)} \varepsilon^{p^*q^*-1}\right)
\end{aligned}
$$
with two positive constants $c_2=c_2\left(c_0,c^u, p^*\right)$ and $c_3=c_3\left(c_0,c^u, c^v, p^*\right)$. We can take a large constant $M>0$ such that $0<c_0<M /8$ to fulfill
$$
\|u\|_{X_1\left(T^*\right)}<\frac{3}{8} M \varepsilon
$$
as long as
$$
8 c_2 M^{-1}\big(\Psi(T^*)\big)^{\beta(p^*,q^*)}\varepsilon^{p^*-1}<1 $$
and
$$ 8 c_3 M^{p^*q^*-1}\big(\Psi(T^*)\big)^{p^*-\beta(p^*,q^*)(p^*q^*-1)} \varepsilon^{p^*q^*-1}<1
$$
hold. Moreover, it follows from \eqref{eq6sec4.2} that
$$
\|v\|_{X_2\left(T^*\right)} \leq \varepsilon\left(c_0+c^v M^{q^*}\big(\Psi(T^*)\big)^{1-\beta(p^*,q^*)q^*}\varepsilon^{q^*-1}\right)<\frac{1}{4} M \varepsilon
$$
when we assume
$$
8 c^v M^{q^*-1}\big(\Psi(T^*)\big)^{1-\beta(p^*,q^*)q^*} \varepsilon^{q^*-1}<1 .
$$
Collecting the above two estimates, we know that
\begin{equation}\label{eq9sec4.2}
F\left(T^*\right)=\|(u, v)\|_{X\left(T^*\right)}=\|u\|_{X_1\left(T^*\right)}+\|v\|_{X_2\left(T^*\right)}<\frac{5}{8} M \varepsilon<M \varepsilon .
\end{equation}
Note that $F=F(T)$ is a continuous function for any $T \in\left(0, T_{\varepsilon}\right)$. From \eqref{eq9sec4.2}, it implies that there exists a time $T_0 \in\left(T^*, T_{\varepsilon}\right)$ satisfying $F\left(T_0\right) \leq M \varepsilon$, which gives a contradiction to the definition of $T^*$. For this reason, we may claim that one of the following estimates are true:
$$
\begin{aligned}
& 8 c_2 M^{-1}\big(\Psi(T^*)\big)^{\beta(p^*,q^*)} \varepsilon^{p^*-1} \geq 1, \\
& 8 c_3 M^{p^*q^*-1}\big(\Psi(T^*)\big)^{p^*-\beta(p^*,q^*)(p^*q^*-1)} \varepsilon^{p^*q^*-1} \geq 1, \\
& 8 c^v M^{q^*-1}\big(\Psi(T^*)\big)^{1-\beta(p^*,q^*)q^*} \varepsilon^{q^*-1} \geq 1 .
\end{aligned}
$$
Hence, one catches
\[
\Psi(T^*)\geq c\varepsilon^{-\min \left\{\frac{p^*-1}{\beta(p^*,q^*)}, \frac{p^*q^*-1}{p^*-\beta(p^*,q^*)(p^*q^*-1)}, \frac{q^*-1}{1-\beta(p^*,q^*)q^*}\right\}}\geq c\varepsilon^{-\frac{q^*(p^*q^*-1)}{q^*+1}},
\]
from which we can find the blow-up time
\begin{equation}\label{eq10sec4.2}
T_{\varepsilon} \geq \Psi^{-1}\left(c \varepsilon^{-\frac{q^*(p^*q^*-1)}{q^*+1}}\right),
\end{equation}
where $c$ is a positive constant independent of $\varepsilon$. \medskip

\noindent \textbf{$\bullet$ Case 2:} The special case $p^*=q^*$ is simple in comparison with Case 1. More precisely, by setting $s(p^*, q^*)=0$ and $\beta(p^*,q^*)=0$ one concludes from \eqref{eq2sec4.2} that
$$
\|v\|_{X_2\left(T^*\right)} \leq \varepsilon\left(c_0+c_1^v M^{p^*} \Psi \left(T^*\right) \varepsilon^{p^*-1}\right) .
$$
By analogous arguments used in \cite{IkeOga2016}, we obtain
\begin{equation}\label{eq11sec4.2}
T_{\varepsilon} \geq \Psi^{-1}\left(c \varepsilon^{-(p^*-1)}\right) .
\end{equation}
At the final step, notice that the function $\Psi(s)$ in \eqref{Psi_function} does not coincide with the function $\psi(s)$ in \eqref{eq2sec4}. For this reason, we need one more step to verify the sharp estimates. To get started, we choose a ``modified'' version of the function $\gamma(t)$, defined in \eqref{con5thr2.1} as follows:
\begin{equation*}\label{eq12sec4.2}
\tilde{\gamma}(t)=\mu_1\tron{C_1(1+t)^{-\frac{1+q^*}{(1-\sigma)(p^*q^*-1)}}}.  
\end{equation*}
We will prove that we can choose $R_1$ large enough such that for $t\geq R_1$, the above function satisfies
\begin{equation}\label{eq13sec4.2}
    \gamma(t)^{-1}\mu_1\tron{C_1(1+t)^{-\frac{1+q^*}{(1-\sigma)(p^*q^*-1)}}}\lesssim C
\end{equation}
and
\begin{equation}\label{eq14sec4.2}
    \gamma(t)^{q^*}\mu_2\tron{C_1(1+t)^{-\frac{1+p^*}{(1-\sigma)(p^*q^*-1)}}\gamma(t)} \lesssim \tron{\mu_1\tron{C_1(1+t)^{-\frac{1+q^*}{(1-\sigma)(p^*q^*-1)}}}}^{\frac{q^*}{q^*+1}}\tron{\mu_2\tron{C_1(1+t)^{-\frac{1+q^*}{(1-\sigma)(p^*q^*-1)}}}}^{\frac{1}{q^*+1}}.
\end{equation}
Indeed, the estimate \eqref{eq13sec4.2} is obvious, so we only verify \eqref{eq14sec4.2}. By the change of variable $s= C_1(1+t)^{-\frac{1+q^*}{(1-\sigma)(p^*q^*-1)}}$, we show that the function
\[
h(s):=\frac{\tron{\mu_1(s)}^{\frac{(q^*)^2}{q^*+1}}\mu_2\tron{C_2s^{\frac{1+p^*}{1+q^*}}\mu_1(s)}}{\tron{\mu_2(s)}^{\frac{1}{q^*+1}}}
\]
is increasing. To be specific, a simple calculation gives
\[
h^{\prime}(s)\geq h(s)\tron{\frac{1+p^*}{1+q^*}\frac{\mu^{\prime}_2\tron{C_2s^{\frac{1+p^*}{1+q^*}}\mu_1(s)}}{\mu_2\tron{C_2s^{\frac{1+p^*}{1+q^*}}\mu_1(s)}}C_2s^{\frac{p^*-q^*}{1+q^*}}\mu_1(s)-\frac{1}{q^*+1}\frac{\mu^{\prime}_2(s)}{\mu_2(s)}}.
\]
Hence, it is sufficent to show that
\begin{equation}\label{eq15sec4.2}
    \frac{\mu^{\prime}_2\tron{C_2s^{\frac{1+p^*}{1+q^*}}\mu_1(s)}}{\mu_2\tron{C_2s^{\frac{1+p^*}{1+q^*}}\mu_1(s)}}C_2s^{\frac{1+p^*}{1+q^*}}\mu_1(s)\geq \frac{\mu^{\prime}_2(s)s}{\mu_2(s)}\text{ for $s\in (0,s_0]$},
\end{equation}
where $s_0$ is a sufficient small constant. To claim \eqref{eq15sec4.2}, we have the following two properties:
\begin{itemize}
    \item $\di\frac{\mu^{\prime}_2(s)s}{\mu_2(s)}$ is non-decreasing on $(0,s_0]$,
    \item $\di \mu_1(s)\geq C^{-1}_2s^{\frac{q^*-p^*}{1+q^*}}$, for $s\in (0,s_0]$.
\end{itemize}
The first one is deduced from the assumption \eqref{con1thr4.1}. Namely, putting 
\[
\omega_1(s):=\di\frac{\mu^{\prime}_2(s)s}{\mu_2(s)},
\]
and using \eqref{con1thr4.1}, one realizes
\[
\omega'_1(s)=\frac{\mu^{\prime}_2(s)}{\mu_2(s)}\tron{1+\frac{\mu^{\prime\prime}_2(s)s}{\mu^{\prime}_2(s)}-\frac{\mu^{\prime}_2(s)s}{\mu_2(s)}}\geq 0
\]
for $s\in (0,s_0]$ if we choose $s_0$ small enough. This is to achieve the first property. In the same way, to prove the second one we denote
\[
\omega_2(s):=\frac{\mu_1(s)}{s^{\frac{q^*-p^*}{1+q^*}}}
\]
and get
\[
\omega'_2(s)=\frac{\mu^{\prime}_1(s)s^{\frac{q^*-p^*}{1+q^*}}-\frac{q^*-p^*}{1+q^*}s^{\frac{q^*-p^*}{1+q^*}-1}}{s^{\frac{2(q^*-p^*)}{1+q^*}}}=\frac{\omega_2(s)}{s}\tron{\frac{\mu^{\prime}_1(s)s}{\mu_1(s)}-\frac{q^*-p^*}{1+q^*}}.
\]
By \eqref{con1thr4.1} again, it implies that $\omega_2(s)$ is decreasing for small enough constant $s_0$, which leads to 
\[
\mu_1(s)\geq C^{-1}_2s^{\frac{q^*-p^*}{1+q^*}},
\] 
where $C^{-1}_2=C_1^{\frac{p^*-q^*}{1+q^*}}$ and $C_1$ can be chosen as large as we want. Therefore, \eqref{eq15sec4.2} is true and we conclude that $h(s)$ is increasing to obtain \eqref{eq14sec4.2}. Finally, \eqref{eq13sec4.2} and \eqref{eq14sec4.2} allow us to replace the function $\gamma(t)$ by $\tilde{\gamma}(t)$. In this way, we derive a new form of $\Psi(t)$ as
\[
\Psi(t)=\int_{0}^{t}(1+t)^{-1}\tron{\mu_1\tron{C_1(1+t)^{-\frac{n-2\sigma}{2(1-\sigma)}}}}^{\frac{q^*}{1+q^*}}\tron{\mu_2\tron{C_1(1+t)^{-\frac{n-2\sigma}{2(1-\sigma)}}}}^{\frac{1}{q^*+1}}\;dt.
\]
To end this problem, carrying out the change of variable $C_0s^{-\frac{n}{2}+\sigma}=C_1(1+\tau)^{-\frac{n-2\sigma}{2(1-\sigma)}}$ one achieves
\[
\Psi(s)=\psi\big(C_4s^{(1-\sigma)^{-1}}-1\big)+C_5,
\]
where $C_4=\tron{\frac{C_0}{C_1}}^{(1-\sigma)^{-1}}$. Then, for $\varepsilon \in (0,\varepsilon_0]$ we obtain the lower bound
\[
C_4 T_{\varepsilon}^{(1-\sigma)^{-1}}-1\geq
\begin{cases}
    \psi^{-1}\tron{c\varepsilon^{-\frac{q^*(p^*q^*-1)}{q^*+1}}} &\text{ if $p^*<q^*$},\\
    \psi^{-1}\tron{c\varepsilon^{-p^*+1}} &\text{ if $p^*=q^*$},
\end{cases}
\]
which is equivalent to 
\[
T_{\varepsilon} \gtrsim \begin{cases}
    \tron{\psi^{-1}\tron{c\varepsilon^{-\frac{q^*(p^*q^*-1)}{q^*+1}}}}^{1-\sigma} &\text{ if $p^*<q^*$},\\
    \tron{\psi^{-1}\tron{c\varepsilon^{-p^*+1}}}^{1-\sigma} &\text{ if $p^*=q^*$}.
\end{cases}
\]
Summarizing, our proof is completed.
\end{proof}

\section{Conclusion and Open problems}
Throughout this paper, we investigate the sharp condition of modulus of continuity that guarantees the global existence of Sobolev solutions. Furthermore, when the solution exists locally and blows up in finite time, we achieve the sharp lifespan estimates for solutions. In future work, we are interested in studying the following weakly coupled system of semilinear structurally damped $\sigma$-evolution equations with critical nonlinearities:
\begin{equation*}
\begin{cases}
    u_{tt}(t,x)+(-\Delta)^{\sigma} u(t,x)+(-\Delta)^{\delta}u_t(t,x)=|v(t,x)|^{p^*}\mu_1(|v(t,x)|),\quad &x\in \mathbb{R}^n,\, t\geq 0,\\
    v_{tt}(t,x)+(-\Delta)^{\sigma} v(t,x)+(-\Delta)^{\delta}v_t(t,x)=|u(t,x)|^{q^*}\mu_2(|u(t,x)|),\quad &x\in \mathbb{R}^n,\, t\geq 0,\\
    u(0,x)= u_0(x),\quad u_t(0,x)= u_1(x),\quad &x\in \mathbb{R}^n,\\
    v(0,x)= v_0(x),\quad v_t(0,x)= v_1(x),\quad &x\in \mathbb{R}^n,
\end{cases}
\end{equation*}
where $\sigma>1$ is any fractional number and $\delta\in [0,\sigma]$. The terms $(-\Delta)^{\sigma}u$ and $(-\Delta)^{\sigma}v$ are quite difficult to handle by our method used in this work. Indeed, when applying the fractional Laplacian $(-\Delta)^{\sigma}$ to the test function, a new domain $$ Q_3:=\{(x,t):|x|\leq R^{\alpha}\text{ and }t\leq R^{\beta}\}$$
arises, where $\alpha=1/(2\sigma)$ and $\beta=(\sigma-\delta)/\sigma$. This domain contains the origin, which leads to a singularity in the integral. To the best of our knowledge, it still remains an open question so far. In a forthcoming paper, we will partially give a positive answer to this problem.

\section*{Acknowledgments}
This research of Tuan Anh Dao is funded by Vietnam National Foundation for Science and Technology Development (NAFOSTED) under grant number 101.02-2025.23. The research of The Anh Cung is funded by Vietnam National Foundation for Science and Technology Development (NAFOSTED) under grant number 101.02-2023.29. This work was completed while the first and third authors were visiting the Vietnam Institute for  Advanced Study in Mathematics (VIASM). They would like to thank VIASM and Hanoi National University of Education for providing a fruitful working environment. \medskip

\textbf{Conflict of interest:} The authors declare no potential conflict of interests.\medskip

\textbf{Data Availability Statement:} My manuscript has no associated data. No new data was created during the study.


\appendix
\begin{lemma}[A generalized Jensen's inequality, see \cite{Pavic2016}] \label{lem2}
Let $\gamma=\gamma(x)$ be a defined and nonnegative function almost everywhere on $\Omega$, provided that $\gamma$ is positive in a set of positive measure. Then, for each convex function $h$ on $\mathbb{R}$ the following inequality holds:
$$
h\left(\dfrac{\displaystyle\int_{\Omega} f(x) \gamma(x) d x}{\displaystyle\int_{\Omega} \gamma(x) d x}\right) \leq \dfrac{\displaystyle\int_{\Omega} h(f(x))\gamma(x) d x}{\displaystyle\int_{\Omega} \gamma(x) d x},
$$
where $f$ is any nonnegative function such that all the above integrals are meaningful.
\end{lemma}

\begin{lemma}[A mapping property in the scale of fractional spaces $\left\{H^s\right\}_{s \in \mathbb{R}}$, see \cite{Triebel1983}]\label{lem5}
Let $\gamma, s \in \mathbb{R}$. Then, the fractional Laplacian
$$
(-\Delta)^\gamma: f \rightarrow(-\Delta)^\gamma f=\left((-\Delta)^\gamma f\right)(x):=\mathfrak{F}^{-1}\left(|\xi|^{2 \gamma} \widehat{f}(\xi)\right)(x)
$$
maps isomorphically the space $H^s$ onto $H^{s-2 \gamma}$.
\end{lemma}

\begin{lemma}[see \cite{Egorov1997}] \label{lem6}
Let $f=f(x) \in H^s$ and $g=g(x) \in H^{-s}$ with $s \in \mathbb{R}$. Then, the following estimate holds:
$$
\left|\int_{\mathbb{R}^n} f(x) g(x) d x\right| \leq\|f\|_{H^s}\|g\|_{H^{-s}} .
$$
\end{lemma}

\begin{lemma}[see \cite{Anhblowup}] \label{lem7}
Let $s \in \mathbb{R}$. Let $v_1=v_1(x) \in H^s$ and $v_2=v_2(x) \in H^{-s}$. Then, the following relation holds:
$$
\int_{\mathbb{R}^n} v_1(x) v_2(x) d x=\int_{\mathbb{R}^n} \widehat{v}_1(\xi) \widehat{v}_2(\xi) d \xi .
$$
\end{lemma}

\end{document}